\documentclass[11pt]{amsart}
\usepackage{amsmath, amsfonts, amsthm, amssymb, multicol, mathtools, dsfont, verbatim}
\usepackage{graphicx}
\usepackage{float, hyperref}
\usepackage{scalerel,stackengine, subcaption}
\usepackage[usenames,dvipsnames]{xcolor}
\usepackage{enumitem}
\usepackage{pgfplots}
\usepgfplotslibrary{fillbetween}
\pgfplotsset{width=10cm,compat=1.9}
\usepackage[square,sort,comma,numbers]{natbib}
\allowdisplaybreaks

\usepackage[square,sort,comma,numbers]{natbib}
\usepackage{fancyhdr}

\makeatletter
\def\@setauthors{%
  \begingroup
  \def\thanks{\protect\thanks@warning}%
  \trivlist
  \centering\footnotesize \@topsep30\p@\relax
  \advance\@topsep by -\baselineskip
  \item\relax
  \author@andify\authors
  \def\\{\protect\linebreak}

  \normalsize\lowercase{\authors}%
  
	\ifx\@empty\contribs
  \else
    ,\penalty-3 \space \@setcontribs
    \@closetoccontribs
  \fi
  \endtrivlist
  \endgroup
}
\def\@settitle{\begin{center}
\LARGE\lowercase{\@title}
  \end{center}%
}
\makeatother

\definecolor{lightblue}{HTML}{2B77A4}
\definecolor{darkred}{HTML}{9E0D0D}
\hypersetup{
	colorlinks=true,
	linkcolor=darkred,
	urlcolor=darkred,
	citecolor=lightblue
}
\numberwithin{equation}{section}

\newtheorem{thm}{Theorem}[section]
\newtheorem{lma}[thm]{Lemma}
\newtheorem{cor}[thm]{Corollary}

\newtheorem{prop}[thm]{Proposition}
\newtheorem{conj}[thm]{Conjecture}

\renewcommand{\epsilon}{\varepsilon}
\newcommand{\eps}{\varepsilon}

\newcommand{\rd}{\mathbb{R}^d}

\renewcommand{\geq}{\geqslant}
\renewcommand{\leq}{\leqslant}

\newcommand{\hd}{\dim_{\textup{H}}}

\newcommand{\fs}{\dim^\theta_{\mathrm{F}}}
\newcommand{\fd}{\dim_{\mathrm{F}}}

\title{$L^p$ averages  of the   Fourier transform  in finite fields}

\author{Jonathan M. Fraser}
\thanks{The  author was  financially supported by a  \emph{Leverhulme Trust Research Project Grant} (RPG-2023-281)  and an \emph{EPSRC Standard Grant} (Fourier analytic techniques in finite fields, EP/Y029550/1).}
\address{University of St Andrews, Scotland}

\begin{document}


\maketitle
\thispagestyle{empty}

\vspace{-5mm}

\begin{abstract}
The  Fourier transform plays a central role in many geometric and combinatorial problems cast in vector spaces over finite fields. If a set admits optimal $L^\infty$ bounds on its Fourier transform (that is, it is a Salem set), then it can often be analysed more easily. However, in many cases obtaining good \emph{uniform} bounds is not possible, even if `most' points admit good pointwise bounds.  Motivated by this, we propose a framework where one systematically studies the  $L^p$ averages of the Fourier transform and keeps track of how good the $L^p$ bounds are as a function of $p$. This captures more nuanced information about a set than, for example, asking whether it is Salem or not. We explore this idea by considering several examples and find that a rich theory emerges.  Further, we provide  various applications of this  approach;  including to sumset type problems, the finite fields distance conjecture, and the problem of counting $k$-simplices inside a given set.  Our typical application is of the form:~if a set admits good  $L^p$ bounds on its Fourier transform, then we are able to make strong geometric conclusions.
\end{abstract}

\tableofcontents

\section{Introduction}

There is a thriving programme of research which takes problems in harmonic analysis or  geometric measure theory and formulates analogues in vector spaces over finite fields, see e.g.~\cite{bourgain2,dvirda,dvir, iosevich, koh, mocktao,wolff}.  The motivation is partly because  arguments will often simplify and this simplification permits a  greater emphasis  on new ideas. On the other hand, in the finite fields setting new `finite phenomena' often emerge---not seen in the classical Euclidean setting---and this leads to a beautiful interplay between combinatorics, analysis, and geometry. 

 In an influential paper, Iosevich and Rudnev \cite{iosevich} initiated a Fourier analytic approach to a discrete analogue of the Falconer  distance set problem   in vector spaces over  finite fields.  In particular, they introduced the concept of a \emph{Salem set} in a vector space over a finite field and completely resolved the distance set problem for such sets.   Inspired by this, in this paper we propose a framework where one systematically studies the  $L^p$ averages of the Fourier transform (not just the $L^\infty$ `average') and keeps track of how good the $L^p$ bounds are as a function of $p$. This captures richer and more nuanced information about a set than, for example, asking whether it is Salem or not. See \eqref{thetasalem} and \eqref{alphadef} for the key new definitions and  Figure \ref{examples} for a basic demonstration of the sort of additional information we can uncover. We motivate this new framework via several examples and applications and find that a rich theory emerges.

We briefly describe the structure of the paper; please also refer to the contents page above.  We begin by setting up our novel framework in Section \ref{framework}.  As part of this we include a brief comparison with the Euclidean setting  (Section \ref{comparisonss}) and  establish some basic theory in Section  \ref{basictheory}.  In Section \ref{examplessec} we go on to consider several examples.  These examples, in part, serve to highlight the richness of the information our framework captures.  The examples we consider  include sets with a product structure (see Theorem \ref{products} and subsequent applications);  various curves and surfaces described  by polynomial equations (see e.g.~Theorem \ref{sphere0} and Theorem \ref{curves}); and Sidon sets (see Corollary \ref{sidonsalem}).  To help put our framework in a wider context, in Section \ref{gensalemsec} we briefly discuss connections with generalised Salem sets (also from \cite{iosevich}) and in Section \ref{randomsec} we establish that generic sets display good $L^p$ behaviour.  From this point on we turn our attention to applications.  We apply the theory to provide  new results concerning sumsets (see Theorem \ref{sumsets} and Corollary \ref{sumsetgood}); a new approach to estimating certain character sums (see Proposition \ref{charsum} and subsequent examples); and new results concerning the finite fields distance conjecture (see Theorem \ref{distancemain}) and the problem of counting $k$-simplices inside a given set (see Theorem \ref{simplicesmain}).

   In addition to the work presented here,   the   $L^p$ approach proposed in this paper has already been used fruitfully to study exceptional projections in finite fields \cite{firdavsprojections} and the finite fields restriction problem  \cite{firdavsrestriction}.  In the latter paper, we improved on a Stein--Tomas style restriction estimate of Mockenhaupt and Tao \cite{mocktao} by replacing the need for $L^\infty$ control of the Fourier transform by $L^p$ control for suitable $p$. We believe there will be many further applications of this approach in the future. 

\subsection{Background:~finite fields and  Fourier transforms}

 Following \cite{iosevich}, we begin by introducing Fourier analysis in vector spaces over finite fields.    Let $\mathbb{F}_q$ be the finite field of $q$ elements where $q$ is a prime power, and $\mathbb{F}_q^d$ be the $d$-dimensional vector space over $\mathbb{F}_q$ for an integer $d \geq 1$.  The \emph{Fourier transform} of a function $f : \mathbb{F}_q^d \to \mathbb{C}$   is the function $\widehat f :  \mathbb{F}_q^d \to \mathbb{C}$ defined by
\[
\widehat f(x) = q^{-d} \sum_{y \in \mathbb{F}_q^d} \chi(-x \cdot y) f(y)
\]
where $\chi: \mathbb{F}_q \to S^1 \subseteq \mathbb{C}$ is a non-trivial additive character. The specific choice of $\chi$ does not play a significant role in what follows.   We will often use Plancherel's Theorem, e.g. \cite[(2.3)]{iosevich}, which states that
\begin{equation} \label{plancherel}
	\sum_{x \in \mathbb{F}_q^d} |\widehat f(x)|^2  = q^{-d} \sum_{x \in \mathbb{F}_q^d} | f(x)|^2.
\end{equation}
We will use the Fourier transform to study sets $E \subseteq \mathbb{F}_q^d$ and the fundamental tool will be the Fourier transform of $E$ which we denote by $\widehat E$ and  identify with the Fourier transform of the indicator function on $E$, denoted by $\mathbf{1}_{E}$.  In particular, 
\[
\widehat E(x) = q^{-d} \sum_{y \in E} \chi(-x \cdot y) .
\]
In the continuous setting, decay of the Fourier transform of a function or measure at infinity implies smoothness or `randomness' in some quantifiable sense, see \cite{mattila}.  However, in the discrete setting there is no `at infinity' and instead one looks for uniform estimates over $\mathbb{F}_q^d$. 

In order to discuss such uniform estimates for the Fourier transform, it is useful to establish some notation. Throughout the paper we write $A \lesssim B$ to mean there is a constant $c$  independent of $q$ for which $A \leq cB$.  Similarly we write $A \gtrsim B$ to mean $B \lesssim A$ and $A \approx B$ if $A\lesssim B$ and $A \gtrsim B$.  One should always think of $q$ as being very large compared to any implicit constants. We also write $\# E$ to denote the cardinality of a set $E$.  It will also be useful to have notation to convey that $A \lesssim B$ does not hold.  For this we will write $A \gg B$ which, therefore, means that for all $c \geq 1$ there exists   $q$ such that $A > cB$.  We  write $a \wedge b = \min\{a,b\}$ and $a \vee b = \max\{a,b\}$ for real numbers $a$ and $b$. 

When seeking uniform bounds for the Fourier transform of a set $E \subseteq \mathbb{F}_q^d$, one always has the trivial estimate
\begin{equation} \label{trivial}
|\widehat E(x) | \leq q^{-d}   (\# E ) 
\end{equation}
for all $x \in \mathbb{F}_q^d$.  On the other hand,  
by Plancherel's Theorem \eqref{plancherel} and the easy fact that 
 \begin{equation} \label{00}
 |\widehat E(0)| = q^{-d}(\#E),
 \end{equation}
  for all sets $E$, 
\[
q^{-d}\# E = \sum_{x \in \mathbb{F}_q^d} |\widehat E(x)|^2  = q^{-2d}(\# E)^2+\sum_{x \in \mathbb{F}_q^d\setminus \{0\}} |\widehat E(x)|^2.
\]
Therefore, the best one can reasonably hope for is 
 \begin{equation} \label{iosevichassumption}
|\widehat E(x) | \lesssim q^{-d} \sqrt{ \# E}
\end{equation}
for all non-zero $x \in \mathbb{F}_q^d$. This estimate  certainly cannot be improved upon if $\#E = o(q^d)$; see Proposition \ref{alwaysbound} for a quantitative statement.  However, we note that when, for example, $E=  \mathbb{F}_q^d$ is the whole space we get
 \begin{equation} \label{fullset}
|\widehat{\mathbb{F}_q^d}(x) | =0
\end{equation}
for all non-zero $x \in \mathbb{F}_q^d$.  As usual, we write $\mathbb{F}_q^* = \mathbb{F}_q \setminus\{0\}$. A trivial but useful thing to keep in mind is that  $\mathbb{F}_q^d \setminus\{0\} \neq (\mathbb{F}_q^*)^d$.  In \cite{iosevich} sets satisfying \eqref{iosevichassumption} are  called \emph{Salem sets}, that is, a Salem set is a set which admits optimal $L^\infty$ bounds for its Fourier transform. 

The definition of a Salem set above is very natural, but it remains (as far as we know) an interesting open problem to determine whether there exist Salem sets of a  prescribed size.  More precisely, for given $\alpha \in (0,d)$, does there exist $E \subseteq \mathbb{F}_q^d$ satisfying $\# E \approx q^\alpha$ and satisfying \eqref{iosevichassumption}?  Surprisingly this is open outside of some special choices of (integer) $\alpha$.  In fact, \cite[Conjecture 1.5]{randomsalem} conjectures that such sets  do not exist for non-integer $\alpha$.  That said, we do know that, for given $\alpha \in (0,d)$,   there exists $E \subseteq \mathbb{F}_q^d$ satisfying $\# E \approx q^\alpha$ and satisfying
 \begin{equation} \label{iosevichassumption}
|\widehat E(x) | \lesssim q^{-d} \sqrt{ \# E \log q} 
\end{equation}
for all non-zero $x \in \mathbb{F}_q^d$, see \cite[Corollary 1.4]{randomsalem} and \cite{hayes} and also the discussion in Section \ref{randomsec}.  Following \cite{randomsalem}, we refer to such sets as \emph{weak Salem} sets and these will be useful for constructing certain examples.  Although it is not the main purpose of this paper, another benefit of our $L^p$ framework is that it provides a quantitative description of sets with good Fourier decay but which fail to be Salem in the strict sense.  In particular, this includes weak Salem sets.

\section{An $L^p$   approach to quantifying Fourier decay} \label{framework}

  The main idea of this paper is to relax the Salem condition by  considering suitable $L^p$ averages of $|\widehat E(x) |$---rather than uniform estimates---and to keep track of  precisely how good the $L^p$ average behaves as a function of $p$.  This will provide richer information about the set $E$. More precisely, for $p \in [1, \infty]$ and $s \in [0,1]$, we  say $E$ is \emph{$(p,s)$-Salem} if
  \begin{equation} \label{thetasalem}
\| \widehat E \|_p \lesssim q^{-d}(\#E)^{1-s}.
 \end{equation}
 Here
  \begin{equation} \label{lpnorm}
 \| \widehat E \|_p =  \Bigg(\frac{1}{q^{d}} \sum_{x \in \mathbb{F}_q^d\setminus\{0\} }  |\widehat E(x)|^{p} \Bigg)^{1/p}
 \end{equation}
for $p \in [1, \infty)$ with the convention that 
 \[
  \| \widehat E \|_\infty = \sup_{x \in \mathbb{F}_q^d \setminus\{0\}}  |\widehat E(x)|.
 \]
We will be interested in finding the range of $s$ such that a given set is $(p,s)$-Salem.  It is clear that the range of possible $s$ is an interval and so there is particular interest in
  \begin{equation} \label{alphadef}
\alpha_E(p) := \sup\{ s :E \text{ is  $(p,s)$-Salem}\}
 \end{equation}
as a function of $p$.  It is important to remove 0 when defining the $L^p$ norm in \eqref{lpnorm}.  By virtue of the simple identity  \eqref{00}, the inclusion of 0 would lead to unnecessary restrictions on the range of possible $s$ and so we obtain strictly more Fourier analytic information by removing it.   

With these definitions in place we immediately see that being Salem and being $(\infty ,1/2)$-Salem are equivalent and that weak Salem sets are $(\infty, s)$-Salem for all $s<1/2$.  It is also trivial to see that all sets are $(p,0)$-Salem for all $p \in [1, \infty]$ (we will, of course, improve upon this observation later by  interpolation, see Corollary \ref{p1/p}).  Moreover, using Plancherel's Theorem \eqref{plancherel},
 \[
  \| \widehat E \|_2 \lesssim  \Bigg(\frac{1}{q^{d}} \sum_{x \in \mathbb{F}_q^d }  |\widehat E(x)|^{2} \Bigg)^{1/2} \approx q^{-d} \sqrt{\# E}  
 \] 
 and so all sets are   $(2,1/2)$-Salem.   Also,  if $E \subseteq \mathbb{F}_q^d$ is $(p_0,s)$-Salem, then it is $(p,s)$-Salem for all $1 \leq p \leq p_0$.  This is a simple consequence of fact that $\| \widehat E \|_p $ is increasing in $p$.  These last two observations combine to tell us that the important  regime to consider is $p>2$.

The goal of this paper is to motivate the study of the function $p \mapsto \alpha_E(p)$.  We will do this by considering numerous examples, exploring general properties of the function, and using the information provided by the function to prove new results in several geometric and combinatorial problems concerning sets $E \subseteq \mathbb{F}_q^d$.

\begin{figure}[H] 
	\centering
	\includegraphics[width=\textwidth]{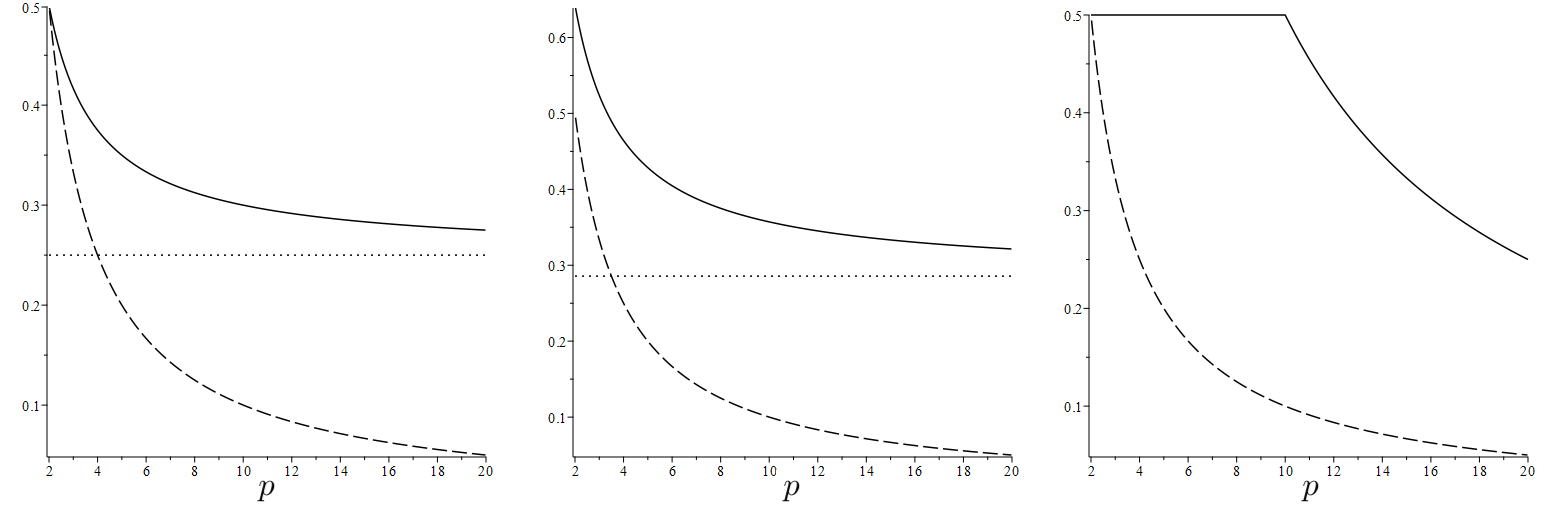}
\caption{Three examples.  Left: the sphere of radius zero in $\mathbb{F}_q^3$ (see Theorem \ref{sphere0}).  Centre: the complement of $\mathbb{F}_q^5$ inside $\mathbb{F}_q^7$  (see Proposition  \ref{smallcomplement}).  Right: a curve in $\mathbb{F}_q^6$ described by 5 linearly independent polynomials (see Theorem \ref{curves}). In each case we plot $\alpha_E(p)$ as a solid curve, the general lower bound from Corollary \ref{p1/p} as a dashed curve, and (in the two cases where $\alpha_E(\infty)>0$) a horizontal asymptote as a dotted line.}\label{examples}
\end{figure}

 \subsection{Comparison with the Euclidean theory} \label{comparisonss}

The new approach in this paper was inspired  by the \emph{Fourier spectrum}, which is a recently introduced concept in Euclidean geometric measure theory that interpolates continuously between the Hausdorff and Fourier dimensions of a (Borel) set.  We refer the reader to \cite{fourierspectrum} and the references therein for more background.  Here we simply record the following analogies:\\

\begin{center}
\begin{tabular}{ |c|c|}  
 \hline
 $E \subseteq \mathbb{F}_q^d$ with $\# E \approx q^\beta$ &  $E \subseteq \mathbb{R}^d$ with Hausdorff dimension $\hd E = \beta$ \\ 
$2 \beta \cdot \alpha_E(\infty) $  & Fourier dimension $\fd E$   \\ 
$ \theta \mapsto 2 \beta \cdot \alpha_E(2/\theta)$ for $\theta \in [0,1]$ & Fourier spectrum $\theta \mapsto \fs E$  for $\theta \in [0,1]$  \\ 
Salem set $E \subseteq \mathbb{F}_q^d$ satisfying \eqref{iosevichassumption} & Salem set $E \subseteq \mathbb{R}^d$ satisfying $\hd E = \fd E$ \\ 
 \hline
\end{tabular}
\end{center}

\vspace{4mm}

There are many important differences between the theory in finite fields and the Euclidean theory.  For example, in the finite fields setting, one is never concerned with measurability or integrability questions, but also cannot rely on throwing away sets of measure zero or using functions with very rapidly decaying Fourier transform (e.g.~Schwartz functions). Moreover, many interesting and counter-intuitive problems arise in the finite setting such as the existence of non-trivial spheres of radius zero (see Section \ref{sphereradiuszero}) and spheres which contain non-trivial affine planes (see Section \ref{gensalemsec}).  Another notable distinction between the finite fields and Euclidean settings is that  when one calls a set $E \subseteq \mathbb{R}^d$ a `Salem set', one genuinely means that this set is intrinsically a Salem set in its own right (relative to its fixed ambient space $\rd$).  This is not the same in the finite fields case, because it is not meaningful to talk about an isolated set $E \subseteq \mathbb{F}_q^d$ since everything is finite.  Instead, we must allow $q$ to vary and we consider the set $E$ as a family of sets parametrised by $q$.  Then, when we say that $E$ is Salem, what we really mean is that the family $\{E\}_q$ is Salem.  As is typical in the literature,  we will suppress this subtly in what follows.
 
 \subsection{First observations} \label{basictheory}
 
In this short subsection we provide some basic but useful facts concerning the $L^p$ averages framework and the function $p \mapsto \alpha_E(p)$.

One reason we expect this approach to provide better results in certain situations is the following concavity estimate. It should be interpreted as showing that  one can never do worse by considering intermediate $p$, but if `strict concavity' is observed then one might do better in the context of a particular application. In particular, the following result says that $p \mapsto p \alpha_E(p)$ is concave. 
 
 \begin{prop} \label{conc}
 If $E \subseteq \mathbb{F}_q^d$ is $(p_0,s_0)$-Salem and $(p_1,s_1)$-Salem for some $1 \leq p_0<p_1 < \infty$, then it is $(p,s)$-Salem for  $p \in [p_0, p_1]$ and 
 \[
 s = s_0 \frac{p_0(p_1-p)}{p(p_1-p_0)} +  s_1 \frac{p_1(p-p_0)}{p(p_1-p_0)}.
 \]
In particular,
\[
\alpha_E(p) \geq 1/p+\alpha_E(\infty)(1-2/p).
\]
 \end{prop}
 
 \begin{proof}
Decomposing $|\widehat E(\cdot)|^p$ according to $p_0, p_1$, we obtain
 \begin{align*}
\| \widehat E \|_p  &=    \Bigg(\frac{1}{q^{d}} \sum_{x \in \mathbb{F}_q^d\setminus\{0\} }  \left(|\widehat E(x)|^{p_0}\right)^{ \frac{p_1-p}{p_1-p_0} } \left(|\widehat E(x)|^{p_1}\right)^{  \frac{p -p_0}{p_1-p_0} }  \Bigg)^{1/p}\\
  &\leq   \Bigg(\frac{1}{q^{d}} \sum_{x \in \mathbb{F}_q^d\setminus\{0\} }  |\widehat E(x)|^{p_0}\Bigg)^{ \frac{p_1-p}{p(p_1-p_0)} }  
   \Bigg(\frac{1}{q^{d}} \sum_{x \in \mathbb{F}_q^d\setminus\{0\} }  |\widehat E(x)|^{p_1}\Bigg)^{  \frac{p -p_0}{p(p_1-p_0)} }\\
    &\, \hspace{5cm}  \qquad \text{(by H\"{o}lder's inequality)}\\
        &=  \Big( \| \widehat E \|_{p_0}\Big)^{ \frac{p_0(p_1-p)}{p(p_1-p_0)} }  
   \Big( \| \widehat E \|_{p_1}\Big)^{  \frac{p_1(p -p_0)}{p(p_1-p_0)} } \\
  &\lesssim q^{-d}(\#E)^{1-s}
  \end{align*}
 by assumption and definition of $s$. This proves the first claim.  The final claim follows by setting $p_0=2$ and letting $p_1 \to \infty$.
 \end{proof}

An immediate consequence of Proposition \ref{conc} is a useful lower bound  $\alpha_E(p) \geq 1/p$ which holds for all sets $E$ and all $p \geq 2$. 

\begin{cor} \label{p1/p}
All sets $E \subseteq \mathbb{F}_q^d$ are $(p,1/p)$-Salem for all $p \in [2, \infty]$.
\end{cor}
 
In the other direction, there is a limit to how good the estimates for intermediate $p$ can be given the estimates for $p=\infty$.  This is the content of the following result, which also serves to establish `continuity' of $\alpha_E(p)$ at $p=\infty$.

  \begin{prop}
 If $E \subseteq \mathbb{F}_q^d$ satisfies 
  \[
  \| \widehat E \|_\infty \gg q^{-d}(\#E)^{1-s},
 \]
 then
 \[
   \| \widehat E \|_p \gg q^{-d}(\#E)^{1-s-\frac{d\log q}{p \log \#E}}
 \]
 for all $p \in [1, \infty)$.   In particular, $\alpha_E(p) \to \alpha_E(\infty)$ as $p \to \infty$. Moreover, if $\# E \approx q^\beta$, then 
\[
\alpha_E(p) \leq \alpha_E(\infty) + \frac{d}{p\beta}.
\]
 \end{prop}
 \begin{proof}
 We have
 \begin{align*}
  \| \widehat E \|_p =  \Bigg(\frac{1}{q^{d}} \sum_{x \in \mathbb{F}_q^d\setminus\{0\} }  |\widehat E(x)|^{p} \Bigg)^{1/p} 
  &\geq \Bigg(\frac{1}{q^{d}} \|\widehat E \|_\infty^{p} \Bigg)^{1/p} \\
    &\gg \Bigg(\frac{1}{q^{d}} q^{-dp}(\#E)^{p(1-s)} \Bigg)^{1/p} \\
    &=  q^{-d}(\#E)^{1-s-\frac{d \log q}{p \log \#E}}
  \end{align*}
  as required. In particular, applying the contrapositive, if $E$ is $(p, s_p)$-Salem for all large  $p$ and $s_p \to s_\infty$ as $p \to \infty$, then $E$  is  $(\infty , s )$-Salem for all $0 \leq s < s_\infty$. Therefore  $\alpha_E(p)$ is upper semi-continuous  at $p=\infty$.     Lower semi-continuity is immediate since $   \| \widehat E \|_p \leq \| \widehat E \|_\infty$ for all $ p \geq 2$ and so $\alpha_E(p)$ is continuous at  $p=\infty$, that is,    $\alpha_E(p) \to \alpha_E(\infty)$ as $p \to \infty$.  The final estimate for $\alpha_E(p) $ follows similarly by applying the contrapositive.
 \end{proof}
 
 We close this introductory section with a simple but useful estimate which gives an explicit lower bound for the $L^p$ norms.
 
\begin{prop} \label{alwaysbound}
  If $E \subseteq\mathbb{F}_{q}^d$ satisfies  $\#E \leq c q^d$ for some $c \in (0,1)$, then for all $p \in [2,\infty]$
  \[
  \| \widehat E \|_p  \geq \sqrt{1- c} \,  q^{- d} \sqrt{\#E } .
 \]
In particular, for such $E$, $\alpha_E(p) \leq 1/2$ for all $p \in [2,\infty]$.
 \end{prop}
 
 \begin{proof} 
 By   \eqref{00} and Plancherel's Theorem \eqref{plancherel} 
 \begin{align*}
 \| \widehat E \|_p^2 \geq  \| \widehat E \|_2^2 &= q^{-d} \sum_{x \in \mathbb{F}_q^d }  |\widehat E(x)|^{2} - q^{-d} |\widehat E (0)|^2 \\
 &=  q^{-2d} (\#E ) - q^{-3d} (\#E )^2 \\
  &=  q^{-2d} (\#E )\big(1- q^{-d} (\#E ) \big) 
 \end{align*}
 and so
 \[
  \| \widehat E \|_p  \geq q^{- d} \sqrt{\#E } \sqrt{1- q^{-d} (\#E )}\geq q^{- d} \sqrt{\#E } \sqrt{1- c}
 \]
which proves the result.
\end{proof}

 \section{Examples} \label{examplessec}

In this section we consider several examples, in part to demonstrate the richness of the theory we are developing, as well as  the varied information which the $L^p$ framework can reveal about a given set.  Some of these examples were already exhibited in Figure \ref{examples}.

In \cite{iosevich} several examples were discussed. In particular, it was shown that the sphere 
\[
S_r^{d-1}= \{ (y_1, \dots, y_{d})  \in \mathbb{F}_{q}^d : y_1^2+\cdots + y_{d}^2 = r\}
\]
is a Salem set of cardinality $\#S_r^{d-1} \approx q^{d-1}$ for $r \in \mathbb{F}^*_q$.  For our purposes, it is more interesting to focus on sets which are \emph{not} Salem.  For such sets, we are interested in how much additional Fourier analytic information can be extracted from studying $L^p$ averages.  That is, we are interested in examples where   $\alpha_E(\infty) < 1/2$ and   $\alpha_E(p)$ exhibits exotic behaviour, for example, beating the estimate  $\alpha_E(p) \geq 1/p+\alpha_E(\infty)(1-2/p)$ coming from Proposition \ref{conc}.

\subsection{The sphere of radius zero} \label{sphereradiuszero}

The sphere of radius zero
\[
S_0^{d-1}= \{ (y_1, \dots, y_{d})  \in \mathbb{F}_{q}^d : y_1^2+\cdots + y_{d}^2 = 0\}
\]
is an interesting object in a vector space over a finite field.  Unlike in the Euclidean case, it usually is a non-trivial set.  In fact, if $d \geq 3$ or if $-1$ is a square in $\mathbb{F}_{q}$, then
\[
\# S_0^{d-1} \approx q^{d-1}
\]
and if $-1$ is not a square in $\mathbb{F}_{q}$, then it is easy to see that $S_0^1 = \{0\}$. See \cite[Proposition 3.1.5]{covert} for more precise calculations.  Also, unlike spheres of non-zero radius, $S_0^{d-1}$ is \emph{not} a Salem set.  However, we prove that (for $d\geq 3$) $S_0^{d-1}$  does have good Fourier analytic behaviour and that this is captured by our $L^p$ approach.

\begin{thm} \label{sphere0}
Suppose that either  $d \geq 3$ or that $d=2$ and $-1$ is a square in $\mathbb{F}_{q}$.  Then  the sphere of radius zero 	$S_0^{d-1} $ is $(p,s)$-Salem if and only if 
	\[
	s \leq \frac{d-2}{2(d-1)}+ \frac{1}{p(d-1)}.
	\]
\end{thm}

\begin{proof}
Writing $z = (z_1, \dots, z_d) \in \mathbb{F}_q^d\setminus \{0\}$, from \cite[Proposition 3.1.6]{covert}
	\[
|	\widehat{S_0^{d-1} }(z) | = q^{-\frac{d+2}{2}} \left\lvert \sum_{x \in \mathbb{F}_q^*} \chi\Big(-\frac{z_1^2+\cdots + z_{d}^2}{4x}\Big) \right\rvert.
	\]
	Therefore,
	\begin{align*}
		\| \widehat{S_0^{d-1} }\|_p &= q^{-d/p-d/2-1} \left( \sum_{z \in \mathbb{F}_q^d\setminus\{0\}}\left\lvert \sum_{x \in \mathbb{F}_q^*} \chi\Big(-\frac{z_1^2+\cdots + z_{d}^2}{4x}\Big) \right\rvert^p\right)^{1/p} \\
		&= q^{-d/p-d/2-1} \left( \sum_{t \in \mathbb{F}_q^*} \sum_{z \in S_t^{d-1}}\left\lvert \sum_{x \in \mathbb{F}_q^*} \chi\Big(-\frac{t}{4x}\Big) \right\rvert^p+\sum_{z \in S_0^{d-1}\setminus\{0\}}\left\lvert \sum_{x \in \mathbb{F}_q^*} \chi(0) \right\rvert^p\right)^{1/p} \\
			&= q^{-d/p-d/2-1} \left( \sum_{t \in \mathbb{F}_q^*} \# S_t^{d-1}+ (\#S_0^{d-1}-1)(q-1)^p\right)^{1/p} \\
				&\approx q^{-d/p-d/2-1} \left( q  q^{d-1}+ q^{d-1}q^p\right)^{1/p} \\
				&\approx q^{-d/p-d/2-1+(d-1)/p+1} \\
				&\approx q^{-d} (\#S_0^{d-1})^{1-s}
		\end{align*}
		for 
\[
s = \frac{d-2}{2(d-1)}+ \frac{1}{p(d-1)}
\]
as required.
\end{proof}

\subsection{Products and direct sums}

We find that sets with a product structure are useful for constructing interesting examples and so we first prove a general result about product sets.

\subsubsection{General theory of products}

\begin{thm} \label{products}
Let $d \geq 2$ and $1 \leq k < d$.  Let $E \subseteq  \mathbb{F}_{q}^{k}$ and $F \subseteq  \mathbb{F}_{q}^{d-k}$ and consider the direct sum $E \oplus F$ which we identify as a subset of $\mathbb{F}_{q}^{d}$. Let $p \geq 1$ and suppose $E$ is $(p,s)$-Salem and $F$ is $(p,t)$-Salem.  Then $E \oplus F$ is $(p,s')$-Salem for 
\[
s'=\frac{s \log  (\#E) + t\log (\#F)}{  \log  (\#E) +  \log (\#F)} \wedge \frac{\frac{k}{p} \log q + t\log (\#F)}{  \log  (\#E) +  \log (\#F)}  \wedge \frac{s \log  (\#E) + \frac{d-k}{p} \log q}{  \log  (\#E) +  \log (\#F)}.
\]
In particular, if $E$ is $(\infty,s)$-Salem and $F$ is $(\infty,t)$-Salem, then $E \oplus F$ is $(\infty,s')$-Salem for 
\[
s'=  \frac{ s \log  (\#E) \wedge  t\log (\#F)    }{  \log  (\#E) +  \log (\#F)}.
\]
\end{thm}

\begin{proof}
For $z \in \mathbb{F}_{q}^{d}$ decompose uniquely as $z=z_k+ z_{d-k}$ for $z_k \in \mathbb{F}_{q}^{k}$ and $z_{d-k} \in \mathbb{F}_{q}^{d-k}$.  Then
\begin{align*}
\widehat{ E \oplus F} (z) &= q^{-d}   \sum_{x \in E  }\sum_{ y   \in F }    \chi(-(x+y) \cdot (z_k+ z_{d-k})) \\
 &= \left(q^{-k}   \sum_{x \in E  } \chi(-x\cdot z_k) \right) \left( q^{-(d-k)} \sum_{ y   \in F }     \chi(-y\cdot z_{d-k})  \right) \\
&=  \widehat{E}(z_k) \widehat F(z_{d-k}).
\end{align*}
Therefore, 
\begin{align*}
\|\widehat{ E \oplus F}  \|_p   &=  \bigg(q^{-d} \sum_{ z_k \in \mathbb{F}_{q}^{k}\setminus \{0\} }  \sum_{ z_{d-k} \in \mathbb{F}_{q}^{d-k}\setminus \{0\}   } | \widehat{E}(z_k) |^p \, |  \widehat F(z_{d-k}) |^p  \\
&\, \qquad 
+   q^{-d}\sum_{ z_{d-k} \in \mathbb{F}_{q}^{d-k}\setminus \{0\}   } | \widehat{E}(0) |^p \, |  \widehat F(z_{d-k}) |^p
 + q^{-d}\sum_{ z_k \in \mathbb{F}_{q}^{k}\setminus \{0\} } | \widehat{E}(z_k) |^p \, |  \widehat F(0) |^p \bigg)^{1/p} \\ 
&\approx   \| \widehat E  \|_p  \| \widehat F \|_p  + q^{-k/p} q^{-k } (\#E)  \| \widehat  F \|_p  + q^{-(d-k)/p} q^{-(d-k)} (\#F) \|  \widehat E  \|_p  \\
&\lesssim q^{-k} (\# E)^{1-s}   q^{-(d-k)} (\# F)^{1-t} \\
&\, \qquad  + q^{-k/p} q^{-k } (\#E)   q^{-(d-k)} (\# F)^{1-t}  + q^{-(d-k)/p} q^{-(d-k)} (\#F)q^{-k} (\# E)^{1-s}  \\
&\approx   q^{-d} (\#E \oplus F)^{1-s'}  
\end{align*}
noting that $\#E \oplus F = (\#E)(\#F)$, for 
\[
s'= \frac{s \log  (\#E) + t\log (\#F)}{  \log  (\#E) +  \log (\#F)} \wedge \frac{\frac{k}{p} \log q + t\log (\#F)}{  \log  (\#E) +  \log (\#F)}  \wedge \frac{s \log  (\#E) + \frac{d-k}{p} \log q}{  \log  (\#E) +  \log (\#F)} 
\]
as required.
\end{proof}

In the above, it might be slightly unsatisfying  to allow $s'$ to depend on $q$.  This can be overcome to provide a simpler statement as long as $\#E$ and $\#F$ are explicit powers of $q$ up to constants. This will certainly be the case in most applications.
\begin{cor} \label{products2}
Let $d \geq 2$ and $1 \leq k < d$.  Let $E \subseteq  \mathbb{F}_{q}^{k}$ and $F \subseteq  \mathbb{F}_{q}^{d-k}$ and suppose $\#E \approx q^\alpha$ and $\#F\approx q^\beta$.  Consider the direct sum $E \oplus F$ which we identify as a subset of $\mathbb{F}_{q}^{d}$ and satisfies $\#(E \oplus F) \approx q^{\alpha+\beta}$. Let $p \geq 1$ and suppose $E$ is $(p,s)$-Salem and $F$ is $(p,t)$-Salem.  Then $E \oplus F$ is $(p,s')$-Salem for 
\[
s'=\frac{s \alpha + t\beta}{  \alpha +  \beta} \wedge \frac{k/p + t\beta}{  \alpha +  \beta}  \wedge \frac{s \alpha + (d-k)/p }{ \alpha +  \beta}.
\]
In particular, if $E$ is $(\infty,s)$-Salem and $F$ is $(\infty,t)$-Salem, then $E \oplus F$ is $(\infty,s')$-Salem for 
\[
s'=  \frac{ s\alpha \wedge  t\beta    }{  \alpha + \beta}.
\]
\end{cor}

Theorem \ref{products} gives information in one direction:~the Fourier behaviour of the product set must be at least as good as the formula stated in terms of the Fourier behaviour of the marginals.  It is also useful to ask for information in the other direction.  This question is slightly more subtle and in general there is not a precise characterisation.  

\begin{thm} \label{productsreverse}
Let $d \geq 2$ and $1 \leq k < d$.  Let $E \subseteq  \mathbb{F}_{q}^{k}$ and $F \subseteq  \mathbb{F}_{q}^{d-k}$ and consider the direct sum $E \oplus F$ which we identify as a subset of $\mathbb{F}_{q}^{d}$. Let $p \geq 1$ and suppose $E$ is \emph{not} $(p,s)$-Salem.  Then $E \oplus F$ is \emph{not} $(p,s')$-Salem for 
\[
s'= \frac{s \log  (\#E) + \frac{d-k}{p} \log q}{  \log  (\#E) +  \log (\#F)}.
\]
Similarly, if $F$ is \emph{not} $(p,t)$-Salem, then $E \oplus F$ is \emph{not} $(p,s')$-Salem for 
\[
s'=  \frac{\frac{k}{p} \log q + t\log (\#F)}{  \log  (\#E) +  \log (\#F)} .
\]
\end{thm}

\begin{proof}
Following the proof of Theorem \ref{products},
\begin{align*}
\|\widehat{ E \oplus F}  \|_p   &\gtrsim   q^{-(d-k)/p} q^{-(d-k)} (\#F) \|  \widehat E  \|_p
\end{align*}
and, since $E$ is \emph{not} $(p,s)$-Salem,
\[
\|  \widehat E  \|_p \gg q^{-k} (\# E)^{1-s}.
\]
Therefore
\begin{align*}
\|\widehat{ E \oplus F}  \|_p   &\gg q^{-(d-k)/p} q^{-(d-k)} (\#F)q^{-k} (\# E)^{1-s}  \\
&=   q^{-d} (\#E \oplus F)^{1-s'}  
\end{align*}
for 
\[
s'=  \frac{s \log  (\#E) + \frac{d-k}{p} \log q}{  \log  (\#E) +  \log (\#F)} 
\]
as required.  The second claim follows by a symmetrical argument which is omitted.
\end{proof}

Combining Theorems \ref{products} and \ref{productsreverse} we get a precise characterisation of $\alpha_{E \oplus F}(p)$ provided the minimum is obtained by either the second or third term.  In particular, this always happens for $p = \infty$.  However, a precise characterisation is not possible in general.  This can be seen  by, for example, defining $E$ and $F$ to alternate between a set with good Fourier behaviour  and a set with bad Fourier behaviour as $q$ increases and choosing $E$ to be good  when $F$ is bad and vice versa.  This means that the (uniform) Fourier behaviour of both $E$ and $F$ in general is bad but the product may have better Fourier behaviour than the formula predicted in Theorem \ref{products}. This is rather artificial and we can rule out this sort of behaviour by making a regularity assumption on either $E$ or $F$. We state the result making an additional assumption on $F$ but the symmetric result using $E$ can be obtained similarly.

\begin{thm} \label{productsreg}
Let $d \geq 2$ and $1 \leq k < d$.  Let $E \subseteq  \mathbb{F}_{q}^{k}$ and $F \subseteq  \mathbb{F}_{q}^{d-k}$ and consider the direct sum $E \oplus F$ which we identify as a subset of $\mathbb{F}_{q}^{d}$. Let $p \geq 1$ and suppose $E$ is \emph{not} $(p,s)$-Salem and 
\[
\|  \widehat F  \|_p \gtrsim q^{-(d-k)} (\# F)^{1-t}.
\]
 Then $E \oplus F$ is \emph{not} $(p,s')$-Salem for 
\[
s'= \frac{s \log  (\#E) + t\log (\#F)}{  \log  (\#E) +  \log (\#F)}.
\]
\end{thm}

\begin{proof}
Following the proof of Theorem \ref{products},
\begin{align*}
\|\widehat{ E \oplus F}  \|_p   &\gtrsim   \| \widehat E  \|_p  \| \widehat F \|_p.
\end{align*}
Since   $E$ is \emph{not} $(p,s)$-Salem,
\[
\|  \widehat E  \|_p \gg q^{-k} (\# E)^{1-s}.
\]
Using this, and the regularity assumption on $F$, 
\begin{align*}
\|\widehat{ E \oplus F}  \|_p   &\gg  q^{-k} (\# E)^{1-s} q^{-(d-k)} (\# F)^{1-t}.  \\
&=   q^{-d} (\#E \oplus F)^{1-s'}  
\end{align*}
for 
\[
s'=   \frac{s \log  (\#E) + t\log (\#F)}{  \log  (\#E) +  \log (\#F)} 
\]
as required. 
\end{proof}

\subsubsection{Applications of general theory of products}

If we set $F = \{0\}$ in the general results above, then we obtain a useful result which characterises Fourier analytic behaviour for sets living in a subspace.

\begin{cor} \label{subspace}
Let $d \geq 2$ and $1 \leq k < d$.  Let $E \subseteq  \mathbb{F}_{q}^{k}$  identified as a subset of $\mathbb{F}_{q}^{d}$. Let $p \geq 1$ and suppose $E$ is $(p,s)$-Salem as a subset of $  \mathbb{F}_{q}^{k}$.  Then   $E$ is $(p,s')$-Salem as a subset of $  \mathbb{F}_{q}^{d}$ if and only if  
\[
s'\leq s \wedge \frac{\frac{k}{p} \log q }{  \log  (\#E) } .
\]
\end{cor}

\begin{proof}
Setting  $F = \{0\} \subseteq \mathbb{F}_{q}^{d-k}$ in Theorem \ref{products} shows that $E$ is $(p,s')$-Salem for $s'\leq s \wedge \frac{\frac{k}{p} \log q }{  \log  (\#E) }$.  To get the `only if' part, we must apply Theorems \ref{productsreverse} and \ref{productsreg} and observe that $F$ is regular in the sense of Theorem  \ref{productsreg} since $\|  \widehat F  \|_p = q^{-(d-k)}$. (The choice of $t$ is irrelevant.)
\end{proof}

Next we consider some cylinder sets, and we observe some non-trivial Fourier behaviour which beats the  $(p,1/p)$ estimate. These sets fail to be Salem, but only just.

\begin{cor}
Let $d \geq 3$, $r \in \mathbb{F}_{q}$ with $r \neq 0$ and 
\[
E = \{ (y_1, \dots, y_{d})  \in \mathbb{F}_{q}^d : y_1^2+\cdots + y_{d-1}^2 = r\}
\]
be the cylinder in $\mathbb{F}_{q}^d$. Then  for all $p \geq 1$ $E$ is $(p, s)$-Salem if and only if
\[
s\leq \frac{2+(d-2)p}{2p(d-1)} .
\]
In particular, $\alpha_E(\infty)= \frac{d-2}{2(d-1)} \to 1/2$ as $d \to \infty$.
\end{cor}

\begin{proof}
Write  $S_r^{d-2} \subset  \mathbb{F}_{q}^{d-1}$ for the  $r$-sphere in $\mathbb{F}_{q}^{d-1}$.  Then $E = S_r^{d-2} \oplus \mathbb{F}_q$ and it follows from Theorem \ref{products} that
$E $ is $(p,s)$-Salem for 
\[
s \leq  \frac{\frac{1}{2} \log  ( S_r^{d-2}) + \frac{1}{p} \log q}{  \log  ( S_r^{d-2}) +  \log (\mathbb{F}_q)} = \frac{(d-2)p+2}{2p(d-1)} .
\]
  Here we use that $S_r^{d-2}$ is a Salem set in  $\mathbb{F}_{q}^{d-1}$  (see \cite{iosevich}) and that $\mathbb{F}_q$ is $(p,t)$-Salem for all $t \geq 0$ as a subset of itself.  The `only if' part comes from Theorem  \ref{productsreverse}.
\end{proof}

Next we  provide simple examples of sets $E$ for which $\alpha_E(\infty)=s$ for a dense set of $s$ in $[0,1]$, thus demonstrating that more information can be gleaned from studying $\alpha_E(\infty)$ alone than only considering the   dichotomy   of being Salem or not Salem.

\begin{prop} \label{smallcomplement}
Let $1\leq k <d$,  $E_k= \mathbb{F}_{q}^k \times\{0\} \subseteq \mathbb{F}_{q}^d$ and $E = \mathbb{F}_{q}^d \setminus E_k$.  Then $E$ is  $(p,s)$-Salem if and only if 
\[
s \leq 1 -\frac{k}{d}+ \frac{k}{pd}.
\]
In particular,  $\alpha_E(\infty) = 1-k/d$ and $E$ is Salem if and only if $k \leq d/2$.
\end{prop}

\begin{proof}
Whilst this set is a product set, we give a direct proof rather than appealing to the previous results.  This is because the proof is very short, but also because we have not yet studied the Fourier behaviour of the marginals in detail. Note that $\#E \approx q^d$ and $\widehat E (z) = 0$ whenever $ z \notin E_k^\perp$ but otherwise $|\widehat E (z)| \approx q^{k-d}$ for $0 \neq z \in E_k^\perp$.  Then
\begin{align*}
\|\widehat E \|_p  &\approx  \left(q^{-d}  \sum_{z \in E_k^\perp} |q^{k-d}  |^{p}  \right)^{1/p}  \\
&\approx q^{-d/p} \left(   q^{(d-k)/p} q^{k-d}  \right)  \\
& \approx q^{-d} (\#E)^{1-s}
\end{align*}
for  $s=1-k/d+k/(pd)$ as required.
\end{proof}

The previous example shows that it is possible for a set to be $(4,1/2)$-Salem whilst not being $(\infty, 1/2)$-Salem.  In particular, this will hold for $d/2<k \leq 2d/3$.  This  example also raises the question of whether good Fourier decay is guaranteed if $\#E \sim  q^d$, recalling that  $\mathbb{F}_{q}^d$ itself is $(\infty,s)$-Salem for all $s \geq 0$.  The following proposition gives  a  quantitative answer to this question.

\begin{prop} \label{bigsets}
If $E \subseteq \mathbb{F}_{q}^d$ is such that $\# (E^c) \lesssim q^{\delta}$ for some $\delta \in (0,d)$, then $E$ is $(\infty,s)$-Salem for all $0 \leq s < 1-\delta/d$.  If,  in addition, $\# (E^c)  \gtrsim q^\gamma$ for some $\gamma \in (0,\delta]$, then $E$ is not $(\infty,s)$-Salem for all $ s > 1-\gamma/(2d)$.
\end{prop}

\begin{proof}
Fix  $x \in \mathbb{F}_{q}^d \setminus\{0\}$.  Using \eqref{fullset},
\[
|\widehat E(x)| = q^{-d} \Big\lvert  \sum_{y \in \mathbb{F}_q^d} \chi(-x \cdot y)  - \sum_{y \in E^c} \chi(-x \cdot y)\Big\rvert   =|\widehat{E^c}(x)|.
\]
Therefore,
\[
|\widehat E(x)|  \leq q^{-d} \# (E^c) = q^{-d} (\# E)^{1-t}
\]
for 
\[
t = 1-\frac{\log \# (E^c) }{\log \# E}.
\]
If  $\# (E^c) \lesssim q^{\delta}$, then   for all $s< 1-\delta/d$, $t>s$ for sufficiently large $q$, proving the first claim. Moreover,  since $\# (E^c) \lesssim q^{\delta}$,  Proposition \ref{alwaysbound} gives
\[
|\widehat E(x)| =|\widehat{E^c}(x)|  \gtrsim q^{-d} \sqrt{\# (E^c)} = q^{-d} (\# E)^{1-t'}
\]
for 
\[
t'= 1-\frac{\log \# (E^c) }{2\log \# E}.
\]
Therefore, if    $ \# (E^c) \gtrsim q^{\gamma}$, then for all $s>1-\gamma/(2d)$,   $t'<s$ for sufficiently large $q$, proving the second claim. 
\end{proof}

The sets constructed in Proposition \ref{smallcomplement} show that the bound $\alpha_E(\infty) \geq 1-\delta/d$ coming from  Proposition \ref{bigsets} is sharp (at least for integer $\delta$).

\subsection{Polynomial curves and surfaces}

In this section we consider curves and surfaces defined by polynomials.  We begin with a `flat' example, where we find the worst Fourier behaviour possible.  This shows the bounds from Corollary \ref{p1/p} are sharp.  

\begin{prop}
Suppose $d = mn$ for integers $m,n \geq 1$. For all $p \geq 2$, the set
\[
E = \{ (k, \dots, k) : k \in \mathbb{F}_{q}^n\} \subseteq \mathbb{F}_{q}^d
\]
is not $(p,s)$-Salem for $s>1/p$. (But it is  $(p, 1/p)$-Salem by Corollary \ref{p1/p}.)
\end{prop}

\begin{proof}
For $x = (x_1, \dots, x_d) \in \mathbb{F}_{q}^d \setminus\{0\}$
\begin{align*}
\widehat E (x) &= q^{-d}   \sum_{k   \in \mathbb{F}_{q}^n}\chi\big(-(k, \dots, k) \cdot (x_1, \dots, x_d)\big)  \\
&= q^{-d} \prod_{j=1}^{n} \sum_{k_j  \in \mathbb{F}_{q}}\chi\Big(-k_j \sum_{i \equiv j \textup{ (mod } n)} x_i\Big) \\
&= q^{n-d} \mathbf{1}_{E'}(x)
\end{align*}
where 
\[
E'= \bigcap_{j=1}^n \left\{  y=(y_1, \dots, y_d) \in \mathbb{F}_{q}^d :  \sum_{i \equiv j \textup{ (mod } n)} y_i= 0\right\}.
\]
Then, noting $\#E' = q^{n(m-1)}$ and $\#E = q^{n}$,
\begin{align*}
\|\widehat E \|_p  = \left(q^{-d} \sum_{x \in E'} q^{(n-d)p} \right)^{1/p} = q^{-d/p} q^{n(m-1)/p} q^{n-d} = q^{-d} (\#E)^{1-1/p}
\end{align*}
as required.
\end{proof}

The main example in the above is the case $n=1$, which was considered in \cite[Example 4.2]{iosevich}.  There it was shown that $E$ is not Salem but is generalised Salem in the case $d=2$.  In fact $E$ is not even generalised Salem for $d\geq 3$.  See Section \ref{gensalemsec} for a discussion of generalised Salem sets.
\begin{cor}
 Let
\[
E = \{ (k, \dots, k) : k \in \mathbb{F}_{q}\} \subseteq \mathbb{F}_{q}^d.
\]
Then $E$ is  $(p,  s )$-Salem if and only if $s \leq 1/p$.
\end{cor}
The example above can be tweaked to obtain better Fourier behaviour.  This can be interpreted as adding curvature.

\begin{prop}
For $d \geq 2$, let
\[
E = \{ (k, \dots, k, k^{-1}) : k \in \mathbb{F}_{q}^*\} \subseteq \mathbb{F}_{q}^d.
\]
If $d=2$, then $E$ is Salem.   On the other hand, if    $d \geq 3$,  $E$ is $(p,2/p)$-Salem for all $p\geq 4$.
\end{prop}

\begin{proof}
Let $p \geq 4$. For $z=(z_1, \dots ,z_d) \in \mathbb{F}_{q}^d$
\begin{align*}
q^d\widehat E (z)  =  \sum_{k \in \mathbb{F}_{q}^*} \chi(-(z_1+ \cdots + z_{d-1}) k -z_d k^{-1})
\end{align*}
is a Kloosterman sum.  Therefore,  applying e.g.~\cite[Theorem 5.45]{nied},
\[
|\widehat E (z)| \leq 2q^{-d} \sqrt{q} \approx q^{1/2-d}
\]
whenever $z_1+ \cdots + z_{d-1} \neq 0$ and $z_d \neq 0$ and $|\widehat E (z)| = q^{1-d}$ otherwise. The former condition will be satisfied for $\approx q^d$ many $z \in  \mathbb{F}_{q}^d$, and the latter condition will be satisfied for precisely $q^{d-2} -1$ many \emph{non-zero} $z \in  \mathbb{F}_{q}^d$.  Therefore,
\begin{align*}
\|\widehat E \|_p  \approx q^{-d/p}\left( q^dq^{(1/2-d)p} + \big( q^{d-2}-1 \big)q^{(1-d)p} \right)^{1/p} \approx   \left\{\begin{array}{cc} q^{-d} q^{1/2} &  d=2 \\
 q^{-d} q^{1-2/p} & d >2\end{array} \right.
\end{align*}
completing the proof.
\end{proof}

Replacing the Kloosterman sums in the previous result with more general character sums we obtain a new general class of Salem sets.

\begin{thm} \label{curves}
For $d \geq 2$, let
\[
E = \{ (f_1(k), \dots, f_d(k)) : k \in \mathbb{F}_{q}\} \subseteq \mathbb{F}_{q}^d
\]
for polynomials $f_1, \dots, f_d$ over $\mathbb{F}_{q}$.  Suppose $f_1, \dots, f_d$  span an $n$-dimensional subspace of $\mathbb{F}_{q}[k]$.  If $d >n$,  then $E$ is $(p,n/p)$-Salem for all $p \geq 2n$.  On the other hand, if  $d=n$, that is, $f_1, \dots, f_d$ are linearly independent polynomials, then $E$ is Salem.
\end{thm}

\begin{proof}
Let $p \geq 2n$. For $z=(z_1, \dots ,z_d) \in \mathbb{F}_{q}^d \setminus\{0\}$
\begin{align*}
q^d\overline{\widehat E (z)}  =  \sum_{k \in \mathbb{F}_{q}} \chi(z_1f_1(k)+ \cdots +  z_d f_d(k))
\end{align*}
is a character sum.   Therefore,  applying Weil's Theorem, e.g. \cite[Theorem 5.38]{nied},
\[
|\widehat E (z)| \leq 2q^{-d} \sqrt{q} \approx q^{1/2-d}
\]
whenever $ z_1f_1(k)+ \cdots + z_d f_d(k)$ is a non-trivial polynomial in $k$.  This will be the case for all but $q^{d-n}-1$ choices of non-zero $z \in  \mathbb{F}_{q}^d$.   On the other hand, if $ z_1f_1(k)+ \cdots + z_d f_d(k)$ is a trivial polynomial, then $|\widehat E (z)| = q^{1-d}$.  Therefore
\begin{align*}
\|\widehat E \|_p  \approx  q^{-d/p}\left( q^dq^{(1/2-d)p} + \big(q^{d-n}-1\big)q^{(1-d)p} \right)^{1/p} \approx   \left\{\begin{array}{cc} q^{-d} q^{1/2} &  d=n \\
 q^{-d} q^{1-n/p} & d >n\end{array} \right.
\end{align*}
completing the proof.
\end{proof}

A well-known and beautiful example falling under Proposition \ref{curves} is the \emph{Veronese curve}. 
\begin{cor}
The rational normal curve (or Veronese curve)
\[
\{(k, k^2, \dots , k^d) : k \in \mathbb{F}_q\}
\]
is a Salem set in $ \mathbb{F}_q^d$.
\end{cor}

\section{Generalised Salem sets} \label{gensalemsec}

Iosevich--Rudnev \cite{iosevich}  also considered a weaker estimate   than \eqref{iosevichassumption}, that is, a condition  weaker than the Salem property.  They termed a set $E \subseteq \mathbb{F}_q^d$ a \emph{generalised Salem set} if for all $\eps>0$ and $t \in \mathbb{F}_q^*= \mathbb{F}_q \setminus \{0\}$
\begin{equation} \label{gensalem}
 \sum_{\substack{m \in \mathbb{F}_q^d:\\ |m|^2=t}} | \widehat E(m)|^2 \lesssim_{\eps} q^\eps q^{-3d/2-1} (\#E)^2.
\end{equation}
Here and below  we use the shorthand 
\[
|m|^2 =  m_1^2+\cdots + m_{d}^2 
\]
for $m =(m_1, \dots, m_d) \in \mathbb{F}_q^d$. In this section we briefly compare the generalised Salem property \eqref{gensalem} with our $L^p$ approach.  To narrow our focus, we first show that generalised Salem sets tend to be large.

\begin{prop}
If $E \subseteq \mathbb{F}_q^d$ with $d \geq 2$ satisfies $\#E \lesssim q^\beta$ for some $\beta<1$, then $E$  cannot be a generalised Salem set. 
\end{prop}

\begin{proof}
	Let $S_0^{d-1} \subseteq \mathbb{F}_q^d$ denote the sphere of radius zero.  We use the estimate $\#S_0^{d-1} \lesssim q^{d-1}$, see \cite{covert}. Applying Plancherel's Theorem \eqref{plancherel}  we get
\begin{align*}
\sum_{t \in \mathbb{F}_q^*} \sum_{\substack{m \in \mathbb{F}_q^d:\\ |m|^2=t}} | \widehat E(m)|^2 &= q^{-d} \#E - \sum_{m\in S_0^{d-1}} | \widehat E(m)|^2 \\
&\geq  q^{-d} \#E - \#S_0^{d-1} q^{-2d} (\#E)^2\\
&\gtrsim q^{-d} \#E
\end{align*}
for sufficiently large $q$ since $\#S_0 \#E  \lesssim q^{d-1}q^\beta = o(q^d)$. Therefore, by the pigeon hole principle, it must hold that for some $t \in \mathbb{F}_q^*$,
\[
 \sum_{\substack{m \in \mathbb{F}_q^d:\\ |m|^2=t}} | \widehat E(m)|^2 \gtrsim \frac{1}{q^{d+1}} \#E .
\]
Therefore, by definition and using that $\# E \lesssim q^\beta$ and $\beta<1 \leq d/2$,  $E$ is not generalised Salem, as required.
\end{proof}

We aim to construct examples of sets with $\#E \gtrsim q^{d/2}$ which fail to be generalised Salem sets but do have non-trivial Fourier decay witnessed by our $L^p$  approach (e.g. $\alpha_E(p)>1/p$). We can do this by modifying an example of Iosevich--Rudnev, see  \cite[Example 4.4]{iosevich}.

Suppose $d$ is odd and $-1$ is a square in $\mathbb{F}_q$.  Then there exists an affine $k$-plane $E' \subseteq  \mathbb{F}_q^d$ with $k>d/2-1$ which is completely contained in a sphere $S_t^{d-1}$ for some $t \neq 0$.  It was shown in \cite[Example 4.4]{iosevich} that such $k$-planes exist and that
\begin{equation} \label{defofEE}
E = \{ x \in \mathbb{F}_q^d : x \cdot y = 0 \ \ \text{for all $y \in E'$}\}
\end{equation}
is not generalised Salem.  Indeed, 
\begin{align*}
| \widehat E(m)| &= \left\{\begin{array}{cc} q^{-d} \# E  &  m \in E' \\
 0 & m \notin E'\end{array} \right.  = \left\{\begin{array}{cc} q^{-k}   &  m \in E' \\
 0 & m \notin E'\end{array} \right. 
\end{align*}
and therefore
\[
 \sum_{\substack{m \in \mathbb{F}_q^d:\\ |m|^2=t}} | \widehat E(m)|^2  = q^{-k}
\]
but $\# E = q^{d-k}$.  It is natural to first try the set $E$ from \eqref{defofEE} itself; the hope being that even though  the spherical average over the sphere $S_t^{d-1}$ does not have good bounds,  when the average is taken over the whole space one might get a better bound.  Perhaps surprisingly, this does not work.

\begin{prop}
The set $E$ from  \eqref{defofEE}  is  $(p,s)$-Salem if and only if $s \leq 1/p$. That is, $E$ has the worst possible Fourier behaviour.
\end{prop}

\begin{proof}
Using that $\widehat E$   vanishes  on the complement of $E'$
\begin{align*}
\| \widehat E \|_p &=  \  \Bigg( \frac{1}{q^d}\sum_{x \in \mathbb{F}^d_q\setminus\{0\}}  |\widehat E(x)|^p \Bigg)^{1/p} \\
&=  \  \Bigg( \frac{1}{q^d}  q^{-kp} (\# E') \Bigg)^{1/p} \\
&= q^{-d} (\# E)^{1-1/p}
\end{align*}
and so $E$ is $(p,s)$-Salem if and only if $s \leq 1/p$. 
\end{proof}
We can improve the situation by taking a subset of $E$ with good decay properties and carefully chosen size.

\begin{prop}
Suppose $d$ is odd and $-1$ is a square in $\mathbb{F}_q$.  Then there exists $k>d/2-1$ such that for all $\alpha \in (0,1/2]$ there exists a set $F \subseteq \mathbb{F}_q^d$ such that $\#F \approx q^{d/2+1/2-\alpha}$, $F$ is not generalised Salem, but 
\[
\alpha_F(p) =  1/2 \wedge \frac{d-k}{p(d-k-\alpha)}.
\]
In particular, this  beats the trivial estimate $\alpha_F(p) \geq 1/p$ for all $p>2$.
\end{prop}

\begin{proof}
Let $E$ be as in  \eqref{defofEE} and  let $F \subseteq E$ be a set with cardinality $\#F \approx q^{d-k-\alpha}$  which is a weak Salem set when viewed as a subset of $E \equiv \mathbb{F}_q^{d-k}$.  Then 
\[
| \widehat F(m)| = q^{-d} \# F \approx q^{-(k+\alpha)} 
\]
for $m \in E'$ but we no-longer have that $\widehat F$ vanishes on the complement of $E'$.   However, the behaviour on $E'$ is enough for the bound 
\[
 \sum_{\substack{m \in \mathbb{F}_q^d:\\ |m|^2=t}} | \widehat F(m)|^2 \gtrsim \# (E') q^{-2(k+\alpha)}  = q^{-k-2\alpha}
\]
where $t \neq 0$ is such that $E' \subseteq S_t^{d-1}$. Combining this with
\[
q^{-3d/2-1} (\#F)^2 \approx q^{-3d/2-1+2(d-k-\alpha)} = q^{d/2-1-2k-2\alpha}
\]
and recalling that $d/2-1<k$, we find that  $F$ is not generalised Salem.   Recalling that $d$ is odd, suppose we have chosen $k$  as small as possible, namely $k=d/2-1/2$.  Therefore,  
\[
\#F \approx q^{d/2+1/2-\alpha} \geq q^{d/2}.
\]
Despite failing to be generalised Salem, $F$ still has some good Fourier decay properties captured by our  $L^p$  approach.  Indeed, observing the product structure of $F$ and applying Theorems \ref{products}, \ref{productsreverse} and \ref{productsreg}, one obtains the conclusion 
\[
\alpha_F(p) =  1/2 \wedge \frac{d-k}{p(d-k-\alpha)},
\]
as required.
\end{proof}

\section{Random sets:~how generic is Fourier decay?} \label{randomsec}

If we select a subset of $\mathbb{F}_q^d$ of a fixed size at random, do we expect good Fourier analytic behaviour?  To answer this question, let $\mathbb{P}$ denote the uniform probability measure on $\mathcal{P}_{q^\alpha}(\mathbb{F}_q^d)$, that is, the set consisting of all subsets of $\mathbb{F}_q^d$ of size $\lfloor q^{\alpha} \rfloor$ for fixed $\alpha \in (0,d)$.  With respect to $\mathbb{P}$, generic sets are \emph{weak Salem}, in the sense that  
\[
  \mathbb{P} \Big( \|\widehat X \|_\infty \leq 3q^{-d}q^{\alpha/2}\sqrt{d \log q}\Big) \to 1
\]
as $q \to \infty$.  This follows from a result of Hayes \cite[Theorem 1.13]{hayes} which    answered a question posed by Babai; see also \cite{randomsalem}.   Using monotonicity of $L^p$ norms, the above estimate implies that for all $p \in [2,\infty)$
\[
\mathbb{P} \Big( \|\widehat X \|_p \leq 3q^{-d}q^{\alpha/2}   \sqrt{d \log q} \Big) \geq   \mathbb{P} \Big( \|\widehat X \|_\infty \leq 3q^{-d}q^{\alpha/2}\sqrt{d \log q}\Big) \to 1.
\] 
We  prove a minor strengthening of this  estimate.  For example, it holds that 
\[
\mathbb{P} \Big( \|\widehat X \|_p \leq q^{-d}q^{\alpha/2} \log \log q \Big) \to 1
\]
as $q \to \infty$, but in fact one can replace the $\log \log q$ term with any increasing unbounded function.

\begin{thm} \label{genericsalem}
Let $C: \mathbb{Z}^+ \to (0,\infty)$ be an arbitrary non-decreasing function with $C(q) \to \infty$ as $q \to \infty$ and $p \in [2,\infty)$.  Then, for $q$ which are prime powers, 
\[
\mathbb{P} \Big( \|\widehat X \|_p > C(q) q^{-d}q^{\alpha/2} \Big)  = O(1/C(q)) \to 0 
\] 
as $q \to \infty$.
\end{thm}

\begin{proof}
Fix $p \in [2,\infty)$ throughout the proof. We first estimate the $p$th moments of the Fourier transform of a  random set $X$ chosen with respect to $\mathbb{P}$ evaluated at a fixed $z \in \mathbb{F}_q^d \setminus\{0\}$.  We get
\begin{align*}
\mathbb{E}( | \widehat X(z) |^p ) &=  \frac{1}{\binom{q^d}{\lfloor  q^\alpha \rfloor }} \sum_{X \in \mathcal{P}_{q^\alpha}(\mathbb{F}_q^d)} \bigg\lvert q^{-d} \sum_{x \in  X} \chi(-x \cdot z ) \bigg\rvert^p  \\
&\leq  q^{-dp}\mathbb{E} \bigg( \bigg\rvert \sum_{y \in Y} y \bigg\rvert ^p\bigg)
\end{align*}
where $Y$ is a multiset consisting of $\lfloor  q^\alpha \rfloor$ points  selected uniformly at random from $\chi(\mathbb{F}_q)$.  To justify the final inequality above, it is sufficient to recognise that the first expectation involves sampling from $\chi(\mathbb{F}_q)$ \emph{without replacement} and the second expectation is \emph{with replacement}.  The required inequality  then follows from   Hoeffding's comparison theorem for sampling without replacement \cite[Theorem 4]{hoeffding}.  The heuristic is that sampling with  replacement can only make things worse because we can choose the same point multiple times which leads to constructive interference, but if we sample without replacement then destructive interference becomes more likely.  In fact the two expectations are comparable up to multiplicative constants (they are even asymptotically equivalent as $q \to \infty$) but we do not require this.  Continuing,
\begin{equation} \label{equicircle}
\mathbb{E} \bigg( \bigg\rvert \sum_{y \in Y} y \bigg\rvert ^p\bigg) =  \mathbb{E} \bigg( \bigg\rvert \sum_{n=1}^{\lfloor  q^\alpha \rfloor} y_n \bigg\rvert ^p\bigg) \lesssim q^{\alpha p /2}
\end{equation}
 where $y_n$ are selected    uniformly at random from $\chi(\mathbb{F}_q)$.   To justify the final inequality in \eqref{equicircle} we estimate the following conditional expectation.  Let $m=\# \chi(\mathbb{F}_q)$, recalling that $\chi(\mathbb{F}_q)$ consists of the $m$th roots of unity.   For an integer $N$ and a large modulus  $r \geq 1$,
\begin{align*}
&\, \hspace{-0.8cm}\mathbb{E} \bigg( \Big\lvert \sum_{n=1}^{N} y_n \Big\rvert ^p \quad \bigg\vert  \quad \Big\rvert \sum_{n=1}^{N-1} y_n \Big\rvert ^p  \leq r \bigg)\\
&\leq \sup_{\theta \in [0,2\pi)} \frac{1}{m} \sum_{k=0}^{m-1} \Big\lvert r^{1/p}e^{i \theta} + e^{2\pi i k/m} \Big\rvert^{p}   \\ 
 &= \sup_{\theta \in [0,2\pi)} \frac{1}{m} \sum_{k=0}^{m-1} \Big(\big( r^{1/p}\cos(\theta) +\cos(2\pi k/m)\big)^2+ ( r^{1/p}\sin(\theta)+\sin(2\pi k/m))^2\Big)^{p/2}   \\
&= \sup_{\theta \in [0,2\pi)}  \frac{1}{m} \sum_{k=0}^{m-1} \big( r^{2/p} +2r^{1/p}\cos(2\pi k/m - \theta) +1\big)^{p/2}   \\
 &= r+ O(r^{1-2/p})
\end{align*}
as $r \to \infty$.  The final equality uses that 
\[
\sum_{k=0}^{m-1}  \cos( 2\pi k/m - \theta )  =0 ,
\]
for all $\theta \in [0,2\pi)$, and so the $O(r^{1-1/p})$ term drops out.  
 Therefore,  the estimate
\begin{equation} \label{uniest}
 \mathbb{E} \bigg( \bigg\rvert \sum_{n=1}^{N} y_n \bigg\rvert ^p\bigg) \leq K N^{\beta}
\end{equation}
holds for some $K\geq 1$ and $\beta \geq 1$ for all $N$  whenever
\[
K N^{\beta} + O(K^{1-2/p} N^{\beta(1-2/p)}) \leq K(N+1)^\beta = KN^\beta+K\beta N^{\beta-1}+ O(N^{\beta-2}).
\]
In particular, \eqref{uniest} will hold for some $K$ for all $N$ whenever $\beta \geq p/2$, which proves \eqref{equicircle}.  Therefore, we have the moment estimate
\[
\mathbb{E}( | \widehat X(z) |^p )  \leq c  q^{-dp}q^{\alpha p/2}
\]
for some $c \geq 1$ independent of $q$ and $z \in  \mathbb{F}_q^d \setminus\{0\}$ (but depending on $p$). With this moment estimate in hand, applying Jensen's inequality and Fubini's theorem,
\[
\Big(\mathbb{E}(\|\widehat X \|_p) \Big)^p \leq \mathbb{E}\bigg(q^{-d} \sum_{z \in  \mathbb{F}_q^d \setminus\{0\}} | \widehat X(z) |^p \bigg) = q^{-d} \sum_{z \in  \mathbb{F}_q^d \setminus\{0\}}\mathbb{E}( | \widehat X(z) |^p ) \leq c q^{-dp}q^{\alpha p/2}
\]
and so 
\begin{equation} \label{ifwegett}
\mathbb{E}(\|\widehat X \|_p) \leq c^{1/p} q^{-d}q^{\alpha /2}.
\end{equation}
Now let $C: \mathbb{Z}^+ \to (0,\infty)$ be an arbitrary non-decreasing function  and suppose
\[
\mathbb{P} \Big( \|\widehat X \|_p > C(q) q^{-d}q^{\alpha/2} \Big)  = \rho_q \in [0,1].
\]
Then, since $\|\widehat X \|_p \geq  0.5 q^{-d}q^{\alpha/2}$ for all $X$ with $\# X = \lfloor q^\alpha \rfloor$ (by Proposition \ref{alwaysbound} and for $q$ sufficiently large),
\[
\mathbb{E}( \|\widehat X \|_p  ) \geq \rho_q C(q) q^{-d}q^{\alpha/2} + 0.5(1-\rho_q)  q^{-d}q^{\alpha/2}.
\]
Combining this with \eqref{ifwegett} yields
\[
\rho_q \leq \frac{c^{1/p}-1/2}{C(q)-1/2}.
\]
Therefore if $C(q) \to \infty$ as $q \to \infty$, then $\rho_q = O(1/C(q)) \to 0$ as $q \to \infty$.
\end{proof}

\section{Sumsets}

 Sumsets are a central object of study  in additive combinatorics; see e.g.~\cite{taovu}. Given non-empty $E_1, \dots, E_k \subseteq \mathbb{F}_q^d$, the \emph{sumset} is defined as
\[
E_1+ \cdots + E_k = \{ x_1+ \cdots + x_k : x_i \in E_i\} \subseteq \mathbb{F}_q^d.
\]
A key problem is to relate the cardinality of the sumset with the cardinalities of the original sets.  Clearly one has the `trivial bounds' $\#(E_1+ \cdots + E_k) \geq \max_i \# E_i$ and 
\[
\#(E_1+ \cdots + E_k)  \lesssim q^{d} \wedge (\# E_1) \cdots (\# E_k)
\]
and these cannot be improved on in general. It is of particular interest to determine conditions under which  the  trivial  lower bound can be improved; especially by a polynomial factor, for example, verifying 
\begin{equation} \label{nontrivyay}
\#(E+E) \gtrsim (\# E)^{1+\eps}
\end{equation}
for some $\eps>0$.

One is often especially interested in the sum of just two sets $E_1+E_2$, or the sum of a set with itself $k$ times $kE = E+ \cdots + E$.  This latter case is sometimes referred to as an \emph{iterated sumset}.  Small sumsets, e.g.~those for which $\#(E_1+ \cdots + E_k) \lesssim \max_i \# E_i$ are often the result of additive structure present in the sets $E_i$ and so if one imposes conditions which forbid additive structure in some sense, then non-trivial lower bounds can sometimes be obtained.  We show that such assumptions can be effectively cast in the setting of our $L^p$ approach.

\begin{thm} \label{sumsets}
If  $E_1, \dots, E_k \subseteq \mathbb{F}_q^d$  are, respectively,   $(2p_k,s_k)$-Salem  for H\"{o}lder conjugates $p_i \geq 1$ with $\sum_{i=1}^k \frac{1}{p_i}= 1$, then
\[
\#(E_1+ \cdots + E_k)  \gtrsim q^{d} \wedge (\# E_1 )^{2s_1} \cdots (\#E_k)^{2s_k}.
\]
\end{thm}

\begin{proof}
We generalise a  counting technique we learned form \cite[Lemma 3.1]{directions}. Define a function $f: \mathbb{F}_q^d \to \mathbb{R}$ by
\[
f(z) = \sum_{x_1+\cdots + x_k=z} \mathbf{1}_{E_1}(x_1) \cdots \mathbf{1}_{E_k}(x_k).
\]
Then
\begin{equation} \label{summing}
\sum_{z \in \mathbb{F}_q^d} f(z) = ( \# E_1 ) \cdots (\# E_k),
\end{equation}
\begin{equation} \label{countingg}
\#\{z : f(z) \neq 0\} = \#(E_1+ \cdots + E_k),
\end{equation}
and
\begin{align} \label{fourrr}
\widehat{ f} (m) &= q^{-d} \sum_{x_1,\dots, x_k\in\mathbb{F}_q^d  } \chi((x_1+\cdots + x_k) \cdot m) \mathbf{1}_{E_1}(x_1) \cdots \mathbf{1}_{E_k}(x_k) \nonumber \\
&= q^{(k-1)d} |\widehat{E_1}(m) | \cdots |\widehat{E_k}(m) |.
\end{align}
Therefore,  
\begin{align*}
(\# E_1 )^2 \cdots (\# E_k)^2 &= \Bigg( \sum_{z \in \mathbb{F}_q^d} f(z) \Bigg)^2 \qquad \text{(by   \eqref{summing}) }\\
&\leq \#( E_1+ \cdots + E_k) \, \sum_{z \in \mathbb{F}_q^d}  f(z)^2\qquad \text{(by Jensen's inequality and \eqref{countingg}) }\\
&= \#( E_1+ \cdots + E_k) \, q^d\sum_{m \in \mathbb{F}_q^d}  |\widehat{ f} (m)|^2 \qquad \text{(by Plancherel's Theorem \eqref{plancherel})}\\
&=  \#( E_1+ \cdots + E_k) \, q^{(2k-1)d} \sum_{m \in \mathbb{F}_q^d}  |\widehat{E_1}(m) |^2 \cdots |\widehat{E_k}(m) |^2\qquad \text{(by   \eqref{fourrr}) }\\
&=  \#( E_1+ \cdots + E_k) \, q^{(2k-1)d}  |\widehat{E_1}(0) |^2 \cdots |\widehat{E_k}(0) |^2\\
&\, \hspace{2cm} + \#( E_1+ \cdots + E_k) \, q^{2kd} q^{-d}\sum_{ m \in \mathbb{F}_q^d\setminus\{0\} }  |\widehat{E_1}(m) |^2 \cdots |\widehat{E_k}(m) |^2\\
&\lesssim \#( E_1+ \cdots + E_k)\, q^{-d} (\#E_1 )^2 \cdots (\#E_k)^2\\
&\, \hspace{2cm} + \#( E_1+ \cdots + E_k)\, q^{2kd} \|\widehat{E_1}\|_{2p_1}^2 \cdots \|\widehat{E_k}(m) \|_{2p_k}^2\\
&\, \hspace{6cm}  \text{(by \eqref{00} and H\"{o}lder's inequality)}\\
&\lesssim  \#( E_1+ \cdots + E_k) \, q^{-d} (\# E_1 )^2 \cdots (\# E_k)^2\\
&\, \hspace{2cm} +\#( E_1+ \cdots + E_k)\, (\# E_1 )^{2(1-s_1)} \cdots (\#E_k)^{2(1-s_k)}\\
\end{align*}
by assumption, and so
\[
\#( E_1+ \cdots + E_k)  \gtrsim q^{d} \wedge (\# E_1 )^{2s_1} \cdots (\#E_k)^{2s_k}
\]
as required.
\end{proof}

Specialising to the case of iterated sumsets, we get the following corollary.  In particular, this provides an (analytic) sufficient condition   for  a set $E$ to generate $\mathbb{F}_q^d$ as an additive semigroup (a purely algebraic conclusion). This result needs $\gtrsim q^\eps$ many generators which may seem inefficient but, for comparison, there are sets as big as $\approx q^{d-1/2}$ which do not generate.  One way of interpreting this result is that if polynomially many points are sufficiently random, then they must generate the whole space.

\begin{cor}
If  $E \subseteq \mathbb{F}_q^d$   is $(2k,s_k)$-Salem, then
\[
\#(kE)  \gtrsim q^{d} \wedge (\# E )^{2ks_k}.
\]
In particular, if   $E$ is $(2k,s_k)$-Salem for some  $k$ with
\begin{equation} \label{growthcondition}
 2 k s_k \geq  \frac{(d-\gamma)\log q}{\log \# E}
\end{equation}
for some $0 \leq \gamma<1/2$, then, for sufficiently large $q$,
\[
kE = \mathbb{F}_q^d
\]
for sufficiently large $k$. (Note that \eqref{growthcondition} is guaranteed for some $k$ for $E$ with $\# E \approx q^\beta$ provided $\alpha_E(p)>\frac{d-1/2}{p\beta}$ for some $p>2$.)
\end{cor}

\begin{proof}
The first claim is an immediate consequence of Theorem \ref{sumsets}. For the second claim, it immediately follows from  \eqref{growthcondition} and the first claim that
\[
\# (kE) \gtrsim q^{d-\gamma}.
\]
Therefore, we may assume $q$ is sufficiently large such that 
\begin{equation} \label{avoids}
\# (kE) > q^{d-1/2} = \#(\mathbb{F}_q^{d-1}) \, \sqrt{q}.
\end{equation}
Additive subgroups of $\mathbb{F}_q$ for $q=p^m$ must have cardinality $p^n$ for $n \vert m$ and so the largest proper additive subgroup of $\mathbb{F}_q$  has cardinality less than $\sqrt{q}$.  Therefore by \eqref{avoids} $kE$ contains no proper additive subgroups and therefore $k'kE = \mathbb{F}_q^d$ for sufficiently large $k'$, proving the claim.
\end{proof}

Next we state a simple corollary concerning sumsets,  difference sets, and directions sets.  Recall the \emph{difference set} of $E\subseteq \mathbb{F}_q^d$ is defined by
\[
E-E = \{ x-y : x,y \in E\} = E+(-E)
\]
and  the \emph{direction set} of $E\subseteq \mathbb{F}_q^d$ is defined by
\[
\mathrm{Dir}(E) =  \big((E-E)\setminus\{0\}\big) /\sim
\]
where $\sim$ is an equivalence relation defined on  $\mathbb{F}_q^d\setminus \{0\}$ given by $u\sim v$ if and only if $u,v$ are in a common 1-dimensional subspace of $\mathbb{F}_q^d$.  In particular, $\#\mathrm{Dir}(E) $ is the number of  distinct directions determined by pairs of points in $E$. 

\begin{cor} \label{sumsetgood}
If  $E \subseteq \mathbb{F}_q^d$   is $(4,s)$-Salem, then
\[
\# (E+E) \gtrsim (\# E)^{4s} \wedge q^d,
\]
\[
\# (E-E) \gtrsim (\# E)^{4s} \wedge q^d
\]
and
\[
\# \mathrm{Dir}(E)  \gtrsim \frac{(\# E)^{4s}}{q} \wedge q^{d-1}.
\]
\end{cor}

\begin{proof}
The first claim follows immediately from Theorem \ref{sumsets}.  The second claim is proved similarly setting $k=2$ and defining $f$ by 
\[
f(z) = \sum_{x_1- x_2=z} E(x_1)  E(x_2).
\]
Alternatively, one may simply set $E_1=E$ and $E_2=-E$. The third claim then follows since $\# \mathrm{Dir}(E)  \geq q^{-1} \# (E-E) $ by the pigeonhole principle.
\end{proof}

 The trivial (but sharp) bounds are:
\[
\# (E+E) \gtrsim  (\# E),
\]
\[
\# (E-E) \gtrsim  (\# E)
\]
and
\[
\# \mathrm{Dir}(E)  \gtrsim q^{-1}\# E  
\]
and it is only meaningful to consider sets $E$ with $\# E = o(q^d)$.  These trivial bounds can be achieved by highly structured sets, such as arithmetic progressions. In particular, we see that  non-trivial bounds analogous to \eqref{nontrivyay} are obtained in all cases in Corollary \ref{sumsetgood}  whenever $\alpha_E(4)>1/4$.  This is rather precise information, since \emph{all} sets satisfy $\alpha_E(4)\geq 1/4$, see Corollary \ref{p1/p}.  Moreover, it is possible to obtain optimal results if  $\alpha_E(4)=1/2$.  Indeed, for such sets, Corollary \ref{sumsetgood} gives
\[
\# (E+E) \approx \#(E-E) \approx   (\# E)^2\wedge q^d .
\]
In fact, it is possible that  $\alpha_E(4)=1/2$ even when  $\alpha_E(\infty)=0$; see e.g.~Theorem \ref{curves} or Corollary \ref{sidonproperty} later.  This shows that the averaging method we use is rather more powerful than simply asking for uniform control of the Fourier transform.

 \section{Sidon sets}

The convolution formula allows one to express the Fourier transform of a convolution as the product of the Fourier transforms.  The   convolution $f \ast f$ of a function $f$ supported on a set $E \subseteq \mathbb{F}_q^d$  is a function supported on $E+E$.  Therefore one might imagine that the convolution formula allows one to relate the Fourier analytic properties of $E$ and $E+E$ directly.  In fact this is not possible because the convolution $\textbf{1}_E \ast \textbf{1}_E$ is not $\textbf{1}_{E+E}$ unless $E$ is a single point.  For $y \in E+F$ let
 \[
g(y) = \#\{ (u,v) \in E \times F \ : \ u+v = y\} \geq 1
 \]
 denote the size of the $y$-fibre.  Naively computing the Fourier transform of $E+F$ one obtains
   \begin{align*}
\lvert \widehat{ E+F}(x) \rvert = q^{-d} \Big\lvert \sum_{y \in \mathbb{F}_q^d} \chi(-x \cdot y) \mathbf{1}_{E+F}(y)  \Big\rvert 
  =q^{-d} \Big\lvert  \sum_{u \in E  }\sum_{v \in   F} \frac{\chi(-x \cdot u) \chi(-x \cdot  v) }{ g(u+v)} \Big\rvert 
    \end{align*}
    and then one might hope to estimate this in terms of 
\[
q^{-d}  \Big\lvert \sum_{u \in E  }   \chi(-x \cdot u)  \Big\rvert  \,  \Big\lvert  \sum_{v \in   F} \chi(-x \cdot  v)  \Big\rvert  = q^d |\widehat E(x) | \, |\widehat F(x)|.
\]
However, to have any hope there would have to be an estimate  of the form
\[
 \Big\lvert  \sum_i c_i z_i\Big\rvert   \lesssim_{\{c_i\}}  \Big\lvert  \sum_i  z_i\Big\rvert  
\]
where $c_i \geq 0 $ are weights and $z_i$ are points on the circle.  However, there clearly cannot be any such estimate in general.

Despite the above, we can make progress if $E$ is a Sidon set and we consider $E+E$.  Recall that $E$ is a \emph{Sidon set} if $g(y)$ is as small as it can be for all  $y \in E+F$, that is, $g(x+x)=1$ and $g(x+y) = 2$ for distinct $x,y \in E$. The value 2 is to account for the unavoidable coincidence $x+y=y+x$.   In particular, for Sidon sets  $\textbf{1}_E \ast \textbf{1}_E$ is very close to  $\textbf{1}_{E+E}$.  Recall that if $E \subseteq \mathbb{F}_q^d$ is Sidon, then $\#E \lesssim q^{d/2}$ and that Sidon sets with  $\#E \gtrsim q^{d/2}$ exist.  

\begin{thm} \label{sidonthm}
 If  $E \subseteq \mathbb{F}_q^d$ is  a Sidon set, then
 \[
2^{-1}q^d  \|  \widehat E \|_{2p}^2  - q^{-d} (\# E) \leq   \| \widehat{ E+E} \|_p  \lesssim   q^d  \|  \widehat E \|_{2p}^2. 
 \]
 In particular, if $\#(E+E)  \leq c q^d$ for some $c \in (0,1)$, then 
 \[
  \| \widehat{ E+E} \|_p  \approx_c  q^d  \|  \widehat E \|_{2p}^2 .
 \]
\end{thm}

\begin{proof}
Arguing as above but using the fact that $E$ is Sidon, for all $x \in \mathbb{F}_q^d$
 \begin{align*}
\lvert \widehat{ E+E}(x) \rvert  =q^{-d}  \Big\lvert  \sum_{y \in E + E} \chi(-x \cdot y)   \Big\rvert  
  &=q^{-d} \Big\lvert  \sum_{u \in E  }   \chi(-x \cdot u)^2  \ + \  \sum_{u,v \in E: u \neq v } \frac{\chi(-x \cdot u) \chi(-x \cdot  v) }{ 2} \Big\rvert \\
    &\leq q^{-d}  \sum_{u \in E  }  1  \ + \ 2^{-1}q^{-d}  \Big\lvert \sum_{u \in E  }\chi(-x \cdot u)\Big\rvert  \,  \Big\lvert \sum_{v \in   E}   \chi(-x \cdot  v)   \Big\rvert \\
   &= q^{-d} (\# E)   \ + \  2^{-1} q^d |\widehat E(x) |^2.
    \end{align*}
    Therefore,
  \begin{align*}
\| \widehat{ E+E} \|_p &\leq  q^{-d} (\# E)   \ + \  2^{-1} q^d  \| (\widehat E)^2\|_p \\
&=q^{-d} (\# E)   \ + \ 2^{-1}  q^d  \|  \widehat E \|_{2p}^2\\
&\lesssim  q^d  \|  \widehat E \|_{2p}^2
    \end{align*}
    by Proposition \ref{alwaysbound} (using the obvious fact that $\# E = o(q^d)$ since $E$ is Sidon).  The lower bound is proven similarly, but using the reverse triangle inequality.  Finally, assuming $\#(E+E)  \leq cq^d$ and applying Proposition \ref{alwaysbound} to $E+E$,
\[
q^d  \|  \widehat E \|_{2p}^2 \lesssim    \| \widehat{ E+E} \|_p  \vee q^{-d} (\# E)  \approx   \| \widehat{ E+E} \|_p  \vee q^{-d} \sqrt{\# (E+E) }   \lesssim_c   \| \widehat{ E+E} \|_p
\]
completing the proof.
\end{proof}

The previous result allows us to relate the Fourier analytic properties of $E$ and $E+E$ directly, which was our initial aim.  

\begin{cor} \label{sidon4}
 Let $E \subseteq \mathbb{F}_q^d$ be a Sidon set and let $p \geq 2$.  If  $E$ is $(2p,s)$-Salem, then $E+E$ is $(p,s)$-Salem.  On the other hand,  if  $E+E$ is $(p,s)$-Salem, then $E$ is $(2p,s \wedge 1/2)$-Salem. 
\end{cor}

\begin{proof}
If  $E$ is $(2p,s)$-Salem, then by Theorem \ref{sidonthm}
\[
\| \widehat{ E+E} \|_p \lesssim q^d \|  \widehat E \|_{2p}^2 \lesssim q^d q^{-2d} (\# E)^{2(1-s)} \approx q^{-d} (\# (E+E) )^{(1-s)} 
\]
and so $E+E$ is $(p,s)$-Salem.  Similarly, if   $E+E$ is $(p,s)$-Salem, then
\[
\|  \widehat E \|_{2p}  \lesssim \sqrt{q^{-d}\| \widehat{ E+E} \|_p}  \vee  q^{-d} \sqrt{\#E} \lesssim \sqrt{q^{-2d} \#(E+E)^{1-s}} \vee  q^{-d} \sqrt{\#E} \approx  q^{-d} (\#E)^{1-s \wedge 1/2} 
\]
and so $E$ is $(2p,s \wedge 1/2)$-Salem.
\end{proof}

A perhaps surprising consequence of the previous result is that   Sidon sets necessarily have non-trivial Fourier analytic behaviour, including optimal $L^4$ estimates and an improvement on the general bounds from Corollary \ref{p1/p} for all $2<p<\infty$. 

\begin{cor} \label{sidonsalem}
 If  $E \subseteq \mathbb{F}_q^d$ is a Sidon set, then  $E$ is $(p,2/p)$-Salem for all $p \geq 4$.
\end{cor}

\begin{proof}
Since  $E+E$ is necessarily $(p,1/p)$-Salem for all $p\geq 2$ by Corollary \ref{p1/p}, it follows from Corollary \ref{sidon4} that $E$ is $(p,2/p \wedge 1/2)$-Salem for all $p \geq 4$. In particular, $E$ is $(p,2/p )$-Salem for all $p \geq 4$.
\end{proof}

From Corollary \ref{subspace} we know that there is a limit to how good the Fourier analytic behaviour can be for sets  living in a subspace.  It turns out that large Sidon sets  have the best possible Fourier behaviour given these constraints.   This also shows that Sidon sets need not be Salem sets. 

\begin{cor} \label{sidonproperty}
 Let $E \subseteq \mathbb{F}_q$ be a Sidon set with $\#E \gtrsim \sqrt{q}$ and embed it as a subset of $\mathbb{F}_q^d$.  Then $E$ is $(p,s)$-Salem if and only if $s \leq 2/p$ for all $p \geq 4$.
\end{cor}

\begin{proof}
Let $p \geq 4$.  Corollary \ref{sidonsalem} gives that $E$ is $(p,s)$-Salem for $s \leq 2/p$ and  (using that  $\#E \gtrsim \sqrt{q}$) Corollary \ref{subspace} gives that $E$ is not  $(p,s)$-Salem  for  $s > 2/p$ for all $p \geq 4$.
\end{proof}


 \section{$L^p$ estimates for character  sums}
 
Let $f : \mathbb{F}_q \to \mathbb{F}_q^d$ be an injective  function and consider the character sum
\[
\mathcal{S}_f(z) =  \sum_{x \in  \mathbb{F}_q } \chi(z \cdot f (x)) 
\]
for $z \in  \mathbb{F}_q^d$.  In general one only has the trivial estimate 
\[
| \mathcal{S}_f(z)| \leq q
\]
but often one can do better for specific choices of $f$ and $z \neq 0$; see e.g.~Kloosterman sums when the domain of $f$ is restricted to  $\mathbb{F}_q^*$ and  $f(x) = (x, x^{-1})$ or Weil's theorem when $f$ is a polynomial \cite{nied}.  It is straightforward to relate such character sums to $L^p$ Fourier bounds. 

\begin{prop} \label{charsum}
 If $f( \mathbb{F}_q) \subseteq \mathbb{F}_q^d$ is  $(p,s)$-Salem, then 
\[
\| \mathcal{S}_f \|_p   \lesssim q^{1-s}.
\]
In particular, if $f( \mathbb{F}_q) \subseteq \mathbb{F}_q^d$ is an $(\infty, s)$-Salem set for some $s>0$, then $| \mathcal{S}_f(z)| \leq q^{1-s}$ for all $z \neq 0$.  
\end{prop}

\begin{proof}
For $z \in  \mathbb{F}_q^d$,
\[
q^d \widehat{f( \mathbb{F}_q) } (z) = \overline{\mathcal{S}_f(z) }.
\]
Therefore, 
\begin{align*}
\| \mathcal{S}_f \|_p  = \left( \frac{1}{q^d} \sum_{z \in  \mathbb{F}_q^d \setminus\{0\} } \left\lvert \mathcal{S}_f(z)  \right\rvert^p \right)^{1/p} &=   \left(   \frac{1}{q^d} \sum_{z \in  \mathbb{F}_q^d \setminus\{0\}} q^{pd}\left\lvert  \widehat{f( \mathbb{F}_q) } (z)   \right\rvert^p \right)^{1/p} \\
&= q^d\| \widehat{ f( \mathbb{F}_q) }\|_p \\
&\lesssim   (\# f( \mathbb{F}_q) )^{1-s} \\
&\approx q^{1-s}
\end{align*}
as required.
\end{proof}

Combining the above result with our work on Sidon sets, we can establish optimal $L^4$ bounds for certain character sums.  More general formulations are possible, but we restrict our attention to a simple example.

\begin{cor}
For $a,b \in \mathbb{F}_q$ define the Weil sum
\[
W(a,b) = \sum_{x \in \mathbb{F}_q} \chi(a x+b x^{2}).
\]
Then
\[
\| W \|_4 = \left( \frac{1}{q^2} \sum_{(a,b) \in  \mathbb{F}_q^2 \setminus\{0\}  }  \left\lvert W(a,b) \right\rvert^4 \right)^{1/4} \lesssim \sqrt{q}.
\]
\end{cor}

\begin{proof}
Let $f : \mathbb{F}_q \to \mathbb{F}_q^2$ be defined by 
\[
f(x) = (x, x^{2})
\]
and then $\| W \|_4 = \| S_f  \|_4$.  The image $f(\mathbb{F}_q)$ is a  Sidon set and satisfies $f(\mathbb{F}_q) = q$ and therefore is $(4,1/2)$-Salem by Theorem \ref{sidonsalem}.  Then by Proposition \ref{charsum}
\[
\| W \|_4  \lesssim \sqrt{q}
\]
as required.  To see why $f(\mathbb{F}_q)$ is a  Sidon set, suppose  $(x,x^{2})+ (y,y^{2}) =(z,z^{2})+ (w,w^{2}) $ for some $x,y,z,w \in \mathbb{F}_q$.  If $x \neq z$, then it is a simple algebraic exercise to show that $x=w$, which proves the claim.
\end{proof}

Of course, it follows by Weil's Theorem \cite[Theorem 5.38]{nied}  that the  stronger $L^\infty$ bound holds for $W$ in the above.  What we offer here is a simple and alternative argument to obtain the weaker $L^4$ estimate.  We can also treat Kloosterman sums in a  similar way.  This is despite the fact that the relevant image is not a Sidon set.  We include the details below.  Again it is well known that  the stronger $L^\infty$ bound holds for Kloosterman sums, but our approach is  rather simpler and different to that presented in, for example,  \cite[Theorem 5.45]{nied}.

\begin{cor}
For $a,b \in \mathbb{F}_q$ define the Kloosterman sum
\[
K(a,b) = \sum_{x \in \mathbb{F}_q^*} \chi(a x+b x^{-1}).
\]
Then
\[
\| K \|_4 = \left( \frac{1}{q^2} \sum_{(a,b) \in  \mathbb{F}_q^2 \setminus\{0\}  }  \left\lvert K(a,b) \right\rvert^4 \right)^{1/4} \lesssim \sqrt{q}.
\]
\end{cor}

\begin{proof}
Let $f : \mathbb{F}_q^* \to \mathbb{F}_q^2$ be defined by 
\[
f(x) = (x, x^{-1})
\]
and extend $f$ to the whole of $\mathbb{F}_q$ by $f(0) = (0,0)$. Then $\| K \|_4 = \| S_f - 1 \|_4 \leq \| S_f   \|_4+1$.  The image $E:=f(\mathbb{F}_q) \subseteq \mathbb{F}_q^2$ is not a Sidon set.  In particular
\[
 (x,x^{-1})+ ((-x),(-x)^{-1}) = (0,0)
\]
for all $x \in \mathbb{F}_q^*$.  This is in some sense the \emph{only} way that the image fails to be Sidon and so we can adapt our approach in the Sidon case to handle this case as well.  To understand the sumset, suppose that $(x,x^{-1})+ (y,y^{-1}) = (a,b)$ for $x,y \neq 0$.  Then
\[
1 = y \cdot y^{-1} = (a-x) \cdot (b-x^{-1})
\]
and so $b \cdot x^2-a \cdot b \cdot  x + a = 0$.  This is enough to deduce that for $u,v \in E$
\[
 g(u,v) := \#\{ (u',v') \in E \times E \ : \ u+v = u'+v'\} =   \left\{\begin{array}{cc} 
2 &  u \neq v, \, u+v \neq (0,0) \\
 1 &  u = v, \,  u+v \neq (0,0) \\
 q &  u+v=(0,0)
\end{array} \right.
\] 
The key thing here is that $g(u,v)=2$ is constant for all but $\approx q \approx \# E$ many pairs $(u,v)$.  Therefore, by the reverse triangle inequality,
 \begin{align*}
\lvert \widehat{ E+E}(x) \rvert  &=q^{-2}  \Big\lvert  \sum_{y \in E + E} \chi(-x \cdot y)   \Big\rvert  \\
    &\geq   2^{-1}q^{-2}  \Big\lvert \sum_{u \in E  }\chi(-x \cdot u)\Big\rvert  \,  \Big\lvert \sum_{v \in   E}   \chi(-x \cdot  v)   \Big\rvert \ - \ q^{-2} O( \#E ) \\
   &=   2^{-1} q^2 |\widehat E(x) |^2 \ - \ q^{-2} O( \#E )  .
    \end{align*}
Therefore
\[
|\widehat E(x) |^4 \lesssim  q^{-4}  \lvert \widehat{ E+E}(x) \rvert^2  \vee q^{-8}( \#E)^2  
\]
and 
 \begin{align*}
\|\widehat E\|_4^4 & \lesssim   q^{-4}\ \|\widehat {E+E} \|_2^2 \vee q^{-8} ( \#E)^2    \\ 
& \lesssim   q^{-4}\ (q^{-2} \sqrt{\#(E + E)})^2 \vee q^{-8} ( \#E)^2  \qquad \text{(by Corollary \ref{p1/p})}  \\ 
& \lesssim    q^{-8} ( \#E)^2.
    \end{align*}
In the above we used that $\#(E + E) \approx (\#E)^2$. We have proved that $\|\widehat E\|_4  \lesssim  q^{-2} \sqrt{\#E} $ and so by Proposition  \ref{charsum}
\[
\| K \|_4  \lesssim \sqrt{q}
\]
as required.
\end{proof}

\section{Spherical averages and more counting problems}

We first prove a general bound for spherical averages of the Fourier transform.  This allows us to connect our $L^p$ approach to several well-known counting problems.  We apply it explicitly in the subsequent subsections to the distance set problem and the problem of counting simplices determined by a given set.  Here and below   we use the shorthand 
\[
|m|^2 =  m_1^2+\cdots + m_{d}^2 
\]
for $m =(m_1, \dots, m_d) \in \mathbb{F}_q^d$.

\begin{lma} \label{sphericalaverages}
Let  $E \subseteq\mathbb{F}_q^d$.  Then
\[
 \Bigg( \sum_{\substack{m \in \mathbb{F}_q^d\setminus\{0\}:\\ |m|^2=0}} |\widehat E(m)|^2 \Bigg)^2+ \sum_{t \in \mathbb{F}_q^*} \Bigg( \sum_{\substack{m \in \mathbb{F}_q^d:\\ |m|^2=t}} |\widehat E(m)|^2 \Bigg)^2  \lesssim  q^{2d-1}  \| \widehat E \|_4^4 .
 \]
\end{lma}
\begin{proof}
In order to apply Jensen's inequality below we recall that
\[
 \#\{ m \in \mathbb{F}_q^d:|m|^2=t\} = \# S_t^{d-1} \approx q^{d-1}
\]
for $t \in \mathbb{F}_q^*$ and (when it is not a single point)
\[
 \#\{ m \in \mathbb{F}_q^d :|m|^2=0\} = \# S_0^{d-1}  \approx q^{d-1}.
\]
Estimating directly 
\begin{align*}
 &\, \hspace{-1.5cm}\Bigg( \sum_{\substack{m \in \mathbb{F}_q^d\setminus\{0\}:\\ |m|^2=0}} |\widehat E(m)|^2 \Bigg)^2+ \sum_{t \in \mathbb{F}_q^*} \Bigg( \sum_{\substack{m \in \mathbb{F}_q^d:\\ |m|^2=t}} |\widehat E(m)|^2 \Bigg)^2 \\
& =  q^{2(d-1)} \sum_{t \in \mathbb{F}_q} \Bigg(\frac{1}{q^{d-1}} \sum_{\substack{m \in \mathbb{F}_q^d\setminus\{0\}:\\ |m|^2=t}} |\widehat E(m)|^2 \Bigg)^2 \\
& \lesssim q^{2(d-1)} \sum_{t \in \mathbb{F}_q} \frac{1}{q^{d-1}} \sum_{\substack{m \in \mathbb{F}_q^d\setminus\{0\}:\\ |m|^2=t}} |\widehat E(m)|^4 \qquad \text{(by Jensen's inequality)}\\
& \leq q^{d-1} \sum_{m \in \mathbb{F}_q^d\setminus\{0\}}   |\widehat E(m)|^4  \\
& =q^{2d-1}  \| \widehat E \|_4^4 
\end{align*}
as required.
\end{proof}

 \subsection{Distance sets}
 
 Given $E \subseteq \mathbb{F}_q^d$, the \emph{distance set} of $E$ is
\[
D(E) = \left\{ \sum_{i=1}^d (x_i-y_i)^2 : (x_1, \dots, x_d), (y_1, \dots, y_d) \in  E \right\} \subseteq \mathbb{F}_q.
\]
A well-known and difficult problem  is to relate the cardinality  of $D(E)$ to the cardinality  of $E$. This is a finite fields analogue of the famous Erd\H{o}s distinct distances problem in discrete geometry and the Falconer distance problem in geometric measure theory, see \cite{distancesets, guthkatz, guth, shmerkinwang}.  More precisely, one formulation of the \emph{finite fields distance conjecture} is:
\begin{conj} \label{maindistanceconj}
If $q$ is odd, $d$ is even,   $E \subseteq \mathbb{F}_q^d$, and  $\# E \geq C q^{d/2}$ with $C$ sufficiently large, then $\# D(E) \gtrsim q$. 
\end{conj}
The assumption that $d$ is even in Conjecture \ref{maindistanceconj} is necessary.  This is due to \cite{hart} where it was shown that the conjecture is false  for odd $d$ even if the exponent $d/2$ is replaced by anything strictly smaller than $(d+1)/2$; see also  \cite{covert,  firdavs}.  This  highlights another interesting distinction between the finite fields case and the Euclidean case.  In particular, the Euclidean analogue of Conjecture \ref{maindistanceconj} is open for all $d \geq 2$ and improvements over the exponent $(d+1)/2$ have been given for all $d \geq 2$. In what follows we do not make any assumption about $d$.

We also note that the assumption that $q$ is odd in the above conjecture is necessary.  For example, if $q=2^m$ for large $m$ and $E=S_0^{d-1}$ is the sphere of radius zero, then $\#E \approx q^{d-1}$ but $D(E) = \{0\}$. This latter fact follows easily  since $(x-y)^2 = x^2+y^2$ holds in fields of characteristic 2. In what follows we assume  that $q$ is odd.  

 Iosevich--Rudnev introduced a Fourier analytic approach to the finite fields distance  conjecture in \cite{iosevich} which made use of discrete analogues of Mattila integrals \cite{mattila}.  They proved that   if  $\# E \geq 3 q^{(d+1)/2}$, then $D(E) = \mathbb{F}_q$.   It is also proved in \cite{iosevich} that if $E \subseteq \mathbb{F}_q^d$ satisfies  $\# E \geq C q^{d/2}$ with $C$ sufficiently large and $E$ is  Salem, then $\# D(E) \gtrsim q$.  This should be interpreted as a discrete analogue of Mattila's result resolving the Euclidean distance set problem for Salem sets (that is, sets with optimal Fourier decay), see \cite{mattila}.  We use our $L^p$ approach to strengthen this latter result of  Iosevich--Rudnev, including a solution to the problem for $(4,1/2)$-Salem sets; a significantly larger family of sets than the Salem sets.

\begin{thm} \label{distancemain}
Let $q$ be odd and suppose  $E \subseteq \mathbb{F}_q^d$ satisfies $\# E \geq C q^{d/2}$ with $C$ sufficiently large.  If  $E  $ is $(4,s )$-Salem, then
\[
\# D(E) \gtrsim \min\left\{q, q^{1-d} (\#E)^{4s} \right\}.
\]
In particular, if $E$ is $(4,1/2)$-Salem, that is, if
\[
 \sum_{y \in \mathbb{F}_q^d} |\widehat E(y)|^{4}   \lesssim q^{-3d} (\# E)^2,
\]
then $\# D(E) \gtrsim q.$
\end{thm}

\begin{proof}
We first estimate the   finite field analogue of the \emph{Mattila integral},  introduced in \cite{iosevich}, which we denote by $\mathcal{M}(E)$.  Indeed, by Lemma \ref{sphericalaverages},
\begin{align*}
\mathcal{M}(E) & := \frac{q^{3d+1}}{(\#E)^4} \sum_{t \in \mathbb{F}_q^*} \Bigg( \sum_{\substack{m \in \mathbb{F}_q^d:\\ |m|^2=t}} |\widehat E(m)|^2 \Bigg)^2 \\
&\lesssim \frac{q^{5d}}{(\#E)^4} \| \widehat E \|_{4}^4 \\
& \lesssim \frac{q^{5d}}{(\#E)^4} \left(q^{-4d}(\#E)^{4(1-s) } \right) \\
&=\frac{q^{d}}{(\#E)^{4s}}.
\end{align*}
Therefore, by \cite[Theorem 1.5]{iosevich},
\[
\#D(E) \gtrsim  q \wedge \frac{q}{\mathcal{M}(E) } \gtrsim q \wedge \frac{(\#E)^{4s}}{q^{d-1}}
\]
completing the proof.
\end{proof}

Combining the example  given below Conjecture \ref{maindistanceconj} concerning even $q$ and Theorem \ref{sphere0} we see that the assumption that $q$ is odd  cannot be removed from Theorem \ref{distancemain}. Since the only place we used this assumption was in the application of \cite[Theorem 1.5]{iosevich}, we also observe that the assumption  cannot be removed there. This was surely known to the authors of \cite{iosevich}, but does not appear to follow from anything written there.  Recall by Theorem \ref{sphere0} that the sphere of radius zero is not Salem in the sense of \cite{iosevich}.

Appealing to Theorem \ref{genericsalem} we get two corollaries which show that ``almost all'' sets satisfy  the finite fields distance set conjecture in a weak sense.

\begin{cor}
Let $q$ be odd.  There exists a constant $c>0$ such that the proportion of sets $E \subseteq \mathbb{F}_q^d$ of size at least $q^{d/2}\log \log q$ which satisfy $\#D(E) \geq cq$ tends to 1 as $q \to \infty$.
\end{cor}

\begin{cor}
Let $q$ be odd.  There exists a constant $C$ such that    the proportion of sets $E \subseteq \mathbb{F}_q^d$ of size at least $Cq^{d/2}$ which satisfy $\#D(E) \geq  \frac{q}{\log \log q}$ tends to 1 as $q \to \infty$.
\end{cor}

The following is another direct corollary of Theorem \ref{distancemain}, which is useful for comparing our result to the state of the art. 

\begin{cor} \label{improveee}
Let $\beta>d/2$ and  $q$ be odd.  If $E \subseteq \mathbb{F}_q^d$ satisfies $\# E \approx q^\beta$ and is $(4,d/(4\beta))$-Salem, then $\# D(E) \gtrsim q.$  
\end{cor} 

For example, to the best of our knowledge the state of the art in dimension 2 is due to \cite{state} where they prove that if $q$ is prime and $E \subseteq\mathbb{F}_q^2$ with  $\# E \gtrsim q^{5/4}$, then $\# D(E) \gtrsim q.$  (In fact they are able to say something stronger concerning the \emph{pinned distance set}.)   Applying Corollary \ref{improveee}, we can improve the exponent 5/4 in $\mathbb{F}_q^2$ whenever $E$ is $(4,s)$-Salem for $s>2/5$.

\subsection{Counting $k$-simplices}

Another counting problem where Fourier analytic techniques bear fruit is that of counting simplices determined by a set $E$.  More precisely, and following \cite{simplices}, for $E \subseteq \mathbb{F}_q^d$ and $1 \leq k \leq d$, write $T_k^d(E) $ to denote the set of congruence classes of $k$-simplices determined by $E$.  Here two $k$-simplices in $E$ with vertices $(x_1, \dots, x_{k+1})$ and  $(y_1, \dots, y_{k+1})$ are congruent if  there exists $A \in \mathbb{F}_q^d \rtimes O_d(\mathbb{F}_q)$ such that $A(x_i) = y_i $ for all $i = 1, \dots, k+1$.  Then 
\[
 \# T_k^d(E) \leq  \# T_k^d(\mathbb{F}_q^d) \approx   q^{\binom{k+1}{2}}
\]
and so one natural question to ask is how big must $E$ be in order to ensure $ \# T_k^d(E) \gtrsim   q^{\binom{k+1}{2}}$.  This problem was investigated in \cite{simplices} and we can use their work combined with our $L^p$ averages approach to obtain the following.

\begin{thm} \label{simplicesmain}
Let  $1 \leq k \leq d$ and $q$ be odd.  If   $E$ is $(4,s)$-Salem, then
\[
 \#  T_k^d(E)  \gtrsim \min\left\{ q^{\binom{k+1}{2}}  , \,  (\#E)^{k-1+4s} q^{\binom{k+1}{2}-kd}\right\}.
\]
In particular, if  $\# E \gtrsim q^{\frac{kd}{k+1}}$ and $E$ is $(4,1/2)$-Salem, then $E$ contains a positive proportion of all possible $k$-simplices up to rotation and translation, that is, $ \#  T_k^d(E) \gtrsim q^{\binom{k+1}{2}}$.
\end{thm}

\begin{proof}
Following the proof of \cite[Theorem 1.5]{simplices}, one obtains the estimate
\[
 \#  T_k^d(E) \geq \frac{(\#E)^{2k+2}q^{\binom{k+1}{2}}}{C_1 (\#E)^{2k+2}+C_2 (\#E)^{k-1}q^{(k+2)d+1} \mathcal{S}(E)}
\]
for   constants $C_1, C_2 \geq 1$ and where
\[
\mathcal{S}(E) := \sum_{\xi \neq 0} |\widehat E(\xi)|^2 \sum_{\eta \neq 0: |\eta|^2 = | \xi |^2} |\widehat E(\eta)|^2 .
\]
In \cite{simplices} one proceeds by using Plancherel's Theorem \eqref{plancherel}  to estimate 
\[
\mathcal{S}(E) \leq \sum_{\xi \neq 0} |\widehat E(\xi)|^2 \sum_{\eta \neq 0 } |\widehat E(\eta)|^2 \leq q^{-2d} (\#E)^2
\]
which gives
\[
 \#  T_k^d(E) \geq \frac{(\#E)^{2k+2}q^{\binom{k+1}{2}}}{C_1 (\#E)^{2k+2}+C_2 (\#E)^{k+1}q^{kd+1} }.
\]
 However, grouping the terms differently and applying  Lemma \ref{sphericalaverages} and our $(4,s)$-Salem assumption, we get
\begin{align*}
\mathcal{S}(E)  &=  \Bigg( \sum_{\substack{m \in \mathbb{F}_q^d\setminus\{0\}:\\ |m|^2=0}} |\widehat E(m)|^2 \Bigg)^2+  \sum_{t \in \mathbb{F}_q^*} \Bigg( \sum_{\substack{m \in \mathbb{F}_q^d:\\ |m|^2=t}} |\widehat E(m)|^2 \Bigg)^2   \\
& \lesssim   q^{2d-1}  \| \widehat E \|_4^4 \\
&\lesssim  q^{-2d-1}  (\#E)^{4(1-s)}
\end{align*}
and so
\begin{align*}
 \# T_k^d(E) & \gtrsim \frac{(\#E)^{2k+2}q^{\binom{k+1}{2}}}{ (\#E)^{2k+2}+  (\#E)^{k-1+4(1-s)}q^{kd}} \\
&\gtrsim  \min\left\{ q^{\binom{k+1}{2}}  , \,  (\#E)^{k-1+4s} q^{\binom{k+1}{2}-kd}\right\},
\end{align*}
as required. 
\end{proof}

For sets which fail to have optimal $L^4$ Fourier bounds, we need to make a stronger assumption about their cardinality to get the optimal conclusion.  This is the content of the next corollary, recalling that every set is $(4,1/4)$-Salem by Corollary \ref{p1/p}.

\begin{cor}
Let  $1 \leq k \leq d$, $q$ be odd,  and  $s \in [1/4,1/2)$.  Suppose  $E \subseteq \mathbb{F}_q^d$ satisfies $\#E \gtrsim q^{\frac{kd}{k-1+4s}}$  and that $E$ is $(4,s)$-Salem.  Then $E$ contains a positive proportion of all possible $k$-simplices up to rotation and translation, that is, $ \#  T_k^d(E) \gtrsim q^{\binom{k+1}{2}}$.
\end{cor}

It is an open problem to determine the smallest $\alpha$ such that   $\#E \gtrsim q^{\alpha}$  guarantees that  $ \#  T_k^d(E) \gtrsim q^{\binom{k+1}{2}}$.  In  \cite[Theorem 1.11]{simplices} it was shown that the optimal $\alpha$ is at least $k-1+1/k$.  An interesting feature of our results  is that we can provide an exponent which beats this bound for sets which are $(4,s)$-Salem for $s$ large enough (we do not  get  the analogous statement  for the distance problem).  For example, setting  $d=k$ we get a better exponent than  $k-1+1/k$ provided
\[
s> \frac{2k^2-2k+1}{4(k^2-k+1)}
\]
noting that the right hand side is strictly smaller than $1/2$ for all $k \geq 1$.  

Once again by appealing to Theorem \ref{genericsalem}, we can  show that ``almost all'' large enough sets contain a positive proportion of all possible $k$-simplices up to rotation and translation. 

\begin{cor}
Let  $1 \leq k \leq d$ and $q$ be odd.  There exists a constant $c>0$ such that the proportion of sets $E \subseteq \mathbb{F}_q^d$ of size at least  $ q^{\frac{kd}{k+1}}\log \log q$ which satisfy $  \#  T_k^d(E) \geq c q^{\binom{k+1}{2}}$  tends to 1 as $q \to \infty$.
\end{cor}


\end{document}